\documentclass{amsart}
\usepackage[utf8]{inputenc}
\usepackage{mathrsfs}
\usepackage{amssymb}
\usepackage{tikz-cd}
\usetikzlibrary{calc}
\usepackage{calc}
\usepackage{cite}
\usepackage{url}

\newtheorem{thm}{Theorem}[section]
\newtheorem{lemma}[thm]{Lemma}
\newtheorem{cor}[thm]{Corollary}
\newtheorem{prop}[thm]{Proposition}
\newtheorem*{thm*}{Theorem}
\newtheorem*{prop*}{Proposition}

\theoremstyle{remark}
\newtheorem{rk}[thm]{Remark}
\newtheorem*{claim}{Claim}
\newtheorem{constr}[thm]{Construction}
\newtheorem{warn}[thm]{Warning}
\newtheorem{ex}[thm]{Example}

\theoremstyle{definition}
\newtheorem{defi}[thm]{Definition}
\newtheorem{lud}[thm]{Definition and Lemma}
\newtheorem{dac}[thm]{Definition and Corollary}

\numberwithin{equation}{section}

\newcommand{\Ho}{\textup{Ho}}
\newcommand{\HO}{\mathscr{H}\kern-2.5pt\textit{o}}
\newcommand{\Yo}{\mathscr{Y}\kern-4pt\textit{o}}
\newcommand{\HOinf}{\ensuremath{\HO_\infty}}
\newcommand{\HOcof}{\ensuremath{\HO_{\textup{cof}}}}
\newcommand{\HOrel}{\ensuremath{\HO_{\textup{rel}}}}
\newcommand{\HOD}{\ensuremath{\HO^D_{\textup{cof}}}}
\newcommand{\HOND}{{{\HO_{\text{N}}^D}}}
\newcommand{\nerve}{\textup{N}}
\newcommand{\h}{\textup{h}}
\newcommand{\cat}[1]{\textbf{\textup{#1}}}
\newcommand{\op}{{\textup{op}}}
\newcommand{\diag}{{\textup{diag}}}
\newcommand{\blank}{{\textup{--}}}
\newcommand{\pr}{{\textup{pr}}}
\newcommand{\Hom}{{\textup{Hom}}}
\newcommand{\id}{{\textup{id}}}
\newcommand{\ev}{{\textup{ev}}}
\newcommand{\im}{{\textup{im}}}
\newcommand{\incl}{{\textup{incl}}}
\newcommand{\colim}{\mathop{\textup{colim}}}
\newcommand{\forget}{\mathop{\textup{forget}}}

\def\twocell[#1]{\arrow[#1, dash, phantom, "\Rightarrow"{scale=1.125, yshift=-.4pt, description, allow upside down, sloped, inner sep=0pt}]}

\begin{document}
\title[Homotopy (Pre-)Derivators$\dots$]{Homotopy (Pre-)Derivators of\\ Cofibration Categories and Quasi-Categories}
\author{Tobias Lenz}
\address{Mathematisches Institut, Rheinische Friedrich-Wilhelms-Universit\"at Bonn, Endenicher Allee 60, 53115 Bonn, Germany \& Max-Planck-Institut für Mathematik, Vivatsgasse 7, 53111 Bonn, Germany}
\email{lenz@math.uni-bonn.de}

\begin{abstract} We prove that the homotopy prederivator of a cofibration category is equivalent to the homotopy prederivator of its associated \emph{quasi-category of frames}, as introduced by Szumi\l{}o. We use this comparison result to deduce various abstract properties of the obtained prederivators.
\end{abstract}
\maketitle

\section*{Introduction}
There are various models for abstract homotopy theory, which roughly can be classified into three approaches.

The oldest approach is \emph{homotopical algebra}, a field of study beginning with Quillen's monograph \cite{quillen}. Here the most basic object of study are categories with a subclass of morphisms, called \emph{weak equivalences}, which we would like to think of as invertible. A classical example are the weak homotopy equivalences of topological spaces or the quasi-isomorphisms of chain complexes. 

However, on their own these categories are almost completely intractable, and so the usual models enhance them with additional structure, for example to \emph{cofibration categories} or \emph{model categories}. These additional structures allow one to do actual calculations, and until now the models from homotopical algebra remain the most suited for this.

The approach currently most investigated is \emph{higher category theory}. Here we study ``categories,'' that also have a notion of homotopies between maps, ``higher homotopies between homotopies'' etc. Usually in this approach compositions (and sometimes also identities) are only defined up to ``a coherent system of (higher) homotopies,'' a notion to be made precise by each model individually. Important examples are \emph{quasi-categories} (also known as \emph{$\infty$-categories}) \cite{joyal, htt} and \emph{complete Segal spaces} \cite{segal-spaces}. These models are the best suited for developing an abstract theory of homotopy (co-)limits, because they have built in by definition a theory of \emph{homotopy coherent} diagrams, whereas in homotopical algebra we usually work with strictly commutative diagrams.

Finally, a third approach is provided by \emph{derivators}---introduced by Grothendieck in \cite{derivateurs} as a solution to the notorious issues of triangulated categories---and their various flavours. A \emph{prederivator} $\mathscr D$ is simply a strict $2$-functor $\cat{Cat}\to\cat{CAT}$ which we should imagine mapping a small category $I$ to the appropriate notion of a ``homotopy category of $I$-shaped diagrams'' in some homotopy-theoretic object captured by $\mathscr D$, together with the various restriction functors. A \emph{derivator} is a prederivator that satisfies some additional axioms, that can be thought of as capturing a basic theory of ``homotopy Kan extensions.'' There are also one-sided variants of this: Grothendieck defines \emph{right derivators} allowing a theory of left homotopy Kan extensions (sic) and dually \emph{left derivators}. Derivators are a powerful tool for establishing formulas involving homotopy (co-)limits.

Now it is a natural question to ask whether all these different models are equivalent. This is indeed true (in a precise sense) for ``all models of higher category theory'' by a result of To\"en, cf.~\cite{toen}. One shouldn't however expect that models from different approaches are in general equivalent, even when interpreted in a sensible manner: it seems for example too strong that all complete and cocomplete quasi-categories arise from model categories (and, more severely, the same is to be expected for adjunctions between them), although this is true under additional assumptions, cf.~\cite[Proposition A.3.7.6]{htt}.

But To\"en's result is even stronger: all the models from higher category theory are equivalent in a way that is ``{essentially unique up to sign}'': more precisely each of these models is a higher categorical object itself, and therefore one can define its \emph{space of automorphisms}. To\"en proved that all these spaces are (up to homotopy) discrete with two points, and for quasi-categories the non-trivial component is represented by the assignment $\mathscr C\mapsto\mathscr C^\op$.

In many cases there are classical constructions relating models to each other (also between different approaches), some of which are shown in the diagram below. In the case of higher category theory they are by the above compatible up to sign and equivalence, and in practical cases they are actually known to be compatible up to equivalence. So we can ask the following question: \emph{are also the remaining preferred constructions compatible up to equivalence?}
\par\bigskip
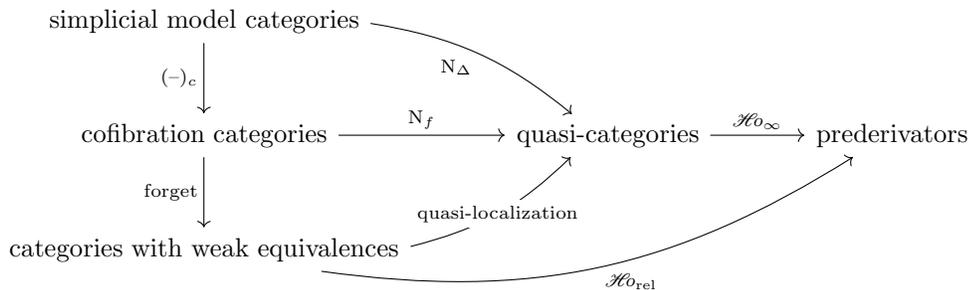
\begin{tikzcd}[row sep=large, column sep=large]
\text{simplicial model categories} \arrow[d, "(\text{--})_c"'] \arrow[rd, "\nerve_\Delta"', bend left=15pt]\\
\text{cofibration categories} \arrow[d, "\text{forget}"'] \arrow[r, "\nerve_f"] & \text{quasi-categories}\arrow[r,"\HOinf"] & \text{prederivators}\\
\text{categories with weak equivalences} \arrow[ru, "\text{quasi-localization}"'{description}, bend right=15pt] \arrow[rru, "\HOrel"', bend right=20pt]
\end{tikzcd}

\subsection*{Own results}
In this paper we investigate the above question in the case of the homotopy prederivators associated to cofibration or quasi-categories. Our main result is the following:

\begin{thm*}
Let $\mathscr C$ be a cofibration category. Then there exists an equivalence $\HOrel(\mathscr C)\simeq\HOinf(\nerve_f(\mathscr C))$, pseudonatural in $\mathscr{C}$.
\end{thm*}

We will then use this result to deduce several properties of the obtained prederivators. In particular we will give a full proof of the following, for which an argument had previously been sketched in \cite{hoder-quasicat}:

\begin{thm*}
Let $\mathscr C$ be a complete and cocomplete quasi-category. Then its homotopy prederivator $\HOinf({\mathscr C})$ is a derivator.
\end{thm*}

\subsection*{Previous results}
Some previous compatibility results of the above type are already known. More specifically, the work of Dwyer and Kan basically shows that the ``big'' triangle on the left commutes up to equivalence, i.e.~for every simplicial model category $\mathscr C$ the quasi-category $\nerve_\Delta({\mathscr C}^\circ)$ is a quasi-localization of the underlying category with weak equivalences, cf.~\cite{dk, dk-modern}. Szumi\l{}o and Kapulkin have proven the analogous result for the lower left triangle, cf.~\cite{frame-loc}.

We remark however, that the lower right triangle does \emph{not} commute, as we will argue in Example~\ref{ex:not-strong}.

\subsection*{Technical assumptions} We assume Grothendieck's Axiom of Universes and fix a sequence $\textbf{U}\in\textbf{V}\in\textbf{W}$ of universes. We will refer to $\textbf{U}$-small sets as ``sets'' and to $\textbf{V}$-small sets as ``large sets''; analogous terminology will be used for simplicial sets etc.

We use the term ``small category'' for a $\textbf{U}$-small category (i.e.~a category with a ``set'' of objects) and ``category'' for a $\textbf{V}$-small category. A $\textbf{W}$-small category will be called a ``large category''; again we will employ analogous terminology for quasi-categories etc.

We remark that our categories are not assumed to be locally small.

\subsection*{Acknowledgements} The results in this article were obtained as part of the author's master thesis at the University of Bonn. He would like to thank his advisor Stefan Schwede for posing this topic, for various helpful discussions, and for several useful remarks on a previous version of this article. Moreover, the author would like to thank the anonymous referee for suggestions improving the exposition of this paper.

This article was prepared for publication while the author was at the Max Planck Institute for Mathematics and he would like to thank them for their hospitality and support.

\section{A short review of abstract homotopy theory}

\subsection{Some terminology} We assume familiarity with the basic theory of quasi-categories as presented in \cite{htt}. We will employ the same terminology and notation as provided there, except that we use the term \emph{quasi-categories} instead of \emph{$\infty$-categories}.

For convenience and definiteness we fix some terminology on categories with weak equivalences:

\begin{defi}
A \emph{category with weak equivalences} is a category $\mathscr C$ equipped with a wide subcategory $W$, called \emph{weak equivalences}, containing all isomorphisms and satisfying 2-out-of-3, i.e.~whenever we have a diagram
\begin{equation*}
A\xrightarrow{f} B\xrightarrow{g} C
\end{equation*}
such that two out of $f,g,gf$ are weak equivalences, then so is the third.

It is called a \emph{homotopical category} if $W$ satisfies 2-out-of-6, i.e.~given a diagram
\begin{equation*}
A\xrightarrow{f} B\xrightarrow{g} C\xrightarrow{h} D
\end{equation*}
such that $hg$ and $gf$ are weak equivalences, then so are $f,g,h,hgf$.

A functor $F\colon\mathscr C\to\mathscr D$ of (the underlying categories of) categories with weak equivalences is called \emph{homotopical} if it sends weak equivalences to weak equivalences.

We denote by \cat{CATWE} the large 2-category of categories with weak equivalences together with homotopical functors and all natural transformations between them, and by \cat{HCAT} the full 2-subcategory of homotopical categories.
\end{defi}

\begin{ex}
Let $\mathscr{C}$ be an ordinary category. Then there are two extreme ways to make $\mathscr{C}$ into a  homotopical category. Namely, on the one hand we can choose the \emph{minimal homotopical structure} and declare precisely the isomorphisms to be weak equivalences; on the other hand we can choose the \emph{maximal homotopical structure} and declare all morphisms to be weak equivalences. We will at several points make use of the latter and we will employ the notation $\widehat{\mathscr{C}}$ for this.
\end{ex}

\begin{ex}\label{ex:levelwise-we}
If $\mathscr{C}$ is a category with weak equivalences, and $I$ is any category, we can equip $\mathscr{C}^I$ with the \emph{levelwise weak equivalences} making it into a category with weak equivalences again, homotopical if $\mathscr{C}$ was. 

If $I$ is actually a category with weak equivalences on its own, we can consider the full subcategory of homotopical functors with the induced structure.

We will tacitly assume, that $\mathscr{C}^I$ is equipped with this structure, unless otherwise noted.
\end{ex}

Since we do not assume any sort of local smallness, any category with weak equivalences $(\mathscr C,W)$ admits a (strict) \emph{localization}, i.e. a functor $\gamma\colon\mathscr C\to\Ho(\mathscr C)$ sending all morphisms of $W$ to isomorphisms such that any other such functor $\mathscr C\to\mathscr D$ factors uniquely through $\gamma$.

\begin{rk}
In the spirit of category theory we should actually only ask for an ``essentially unique factorization up to isomorphism.'' However, the above variant simplifies notation and has the advantage that it gives rise to a \emph{strict} 2-functor $\Ho\colon\cat{CATWE}\to\cat{CAT}$ which will be relevant later (cf.~Definition~\ref{defi:homotopy-der}).
\end{rk}

\subsection{ABC cofibration categories}
\begin{defi}[R\u{a}dulescu-Banu]
An \emph{ABC cofibration category} is a category $\mathscr C$ together with two classes of morphisms, called \emph{weak equivalences} respectively \emph{cofibrations}, satisfying the following axioms (where the term \emph{acyclic cofibration} refers to a map that is both a cofibration and a weak equivalence):
\begin{enumerate}
\item $\mathscr C$ has an initial object $\varnothing$. We call an object $X$ \emph{cofibrant} if the unique map $\varnothing\to X$ is a cofibration. The chosen initial object $\varnothing$ is cofibrant. Cofibrations are stable under composition. All isomorphisms of $\mathscr C$ are weak equivalences. Moreover, if $f\colon X\to Y$ is an isomorphism such that $X$ is cofibrant, then $f$ is an acyclic cofibration (in particular the notion of cofibrancy is independent of the choice of initial object).
\item Weak equivalences satisfy $2$-out-of-$3$.
\item Given a solid arrow diagram
\begin{equation*}
\begin{tikzcd}
A \arrow[d,tail,"i"']\arrow[dr, phantom, "\ulcorner", very near end] \arrow[r] & C \arrow[d, dashed, "j"]\\
B \arrow[r, dashed] & D
\end{tikzcd}
\end{equation*}
with $A, C$ cofibrant and $i$ a cofibration, the pushout exists. Moreover, $j$ is a cofibration, acyclic if $i$ is.
\item Any morphism $f\colon X\to Y$ with $X$ cofibrant can be factored as $f=pi$ with $p$ a weak equivalence and $i$ a cofibration.\label{item:abc-factorization}
\item If $I$ is a set and $(f_i\colon X_i\to Y_i)$ is a family of maps with all $X_i$ cofibrant and all $f_i$ cofibrations, then the coproducts $\coprod_{i\in I} X_i,\coprod_{i\in I}Y_i$ exist. Moreover, $\coprod_{i\in I} f_i$ is a cofibration, acyclic if all the $f_i$ are.
\item Given a countable sequence
\begin{equation*}
\begin{tikzcd}
X_0 \arrow[r, tail, "f_0"]  & X_1 \arrow[r, tail, "f_1"] & X_2 \arrow[r, tail] & \cdots
\end{tikzcd}
\end{equation*}
with $X_0$ cofibrant and all $f_i$ cofibrations, the colimit $X_\infty$ exists. Moreover, the induced map $X_0\to X_\infty$ is a cofibration, acyclic if all the $f_i$ are.
\end{enumerate}
Dually, \emph{ABC fibration categories} are defined.
\end{defi}

The above notion is defined and extensively studied in \cite{cof}, where it appears as Definition 1.1.1, generalizing previous notions by Anderson \cite{anderson}, Brown \cite{brown}, and others. It is also studied under the name of a \emph{catégorie dérivable à droite homotopiquement cocomplète} in \cite{approximation}.

While the above generality will be needed later, we are mostly interested in the following special case:

\begin{defi}
A \emph{cofibration category} is an ABC cofibration category with all objects cofibrant. Dually, \emph{fibration categories} are defined.
\end{defi}

\begin{warn}\label{warn:cof}
This is stronger than the classical notion of cofibration categories, which omits the last two axioms.
\end{warn}

\begin{ex}
If $\mathscr{C}$ is an ABC cofibration category, then its full subcategory $\mathscr{C}_c$ of cofibrant objects is a cofibration category.
\end{ex}

\begin{prop}\label{prop:cof-approx-basis}
Let $\mathscr C$ be an ABC cofibration category. Then $\mathscr{C}_c\hookrightarrow\mathscr{C}$ descends to an equivalence $\Ho(\mathscr{C}_c)\to\Ho(\mathscr{C})$.
\begin{proof}
Cf.~\cite[Theorem 6.1.6-(1)]{cof}.
\end{proof}
\end{prop}

\begin{ex}\label{ex:model-abc}
Let $\mathscr C$ be a model category. Then $\mathscr C$ becomes an ABC cofibration category by forgetting the fibrations; in particular the subcategory ${\mathscr C}_c$ of cofibrant objects becomes a cofibration category in the same way, cf.~\cite[Proposition 2.2.4]{cof}.
\end{ex}

We now turn our attention to morphisms of cofibration categories:

\begin{defi}
Let $\mathscr C,\mathscr D$ be cofibration categories. A functor $F\colon\mathscr C\to\mathscr D$ of their underlying categories is called \emph{exact} if the following holds:
\begin{enumerate}
\item $F$ preserves cofibrations and acyclic cofibrations.
\item $F$ preserves arbitrary small coproducts and both (countable) sequential colimits and pushouts along cofibrations.
\end{enumerate}
We denote by \cat{COFCAT} the large category of cofibration categories with the exact functors as morphisms.
\end{defi}

As for model categories we have:

\begin{lemma}[Ken Brown's Lemma]
Any exact functor of cofibration categories is homotopical.
\begin{proof}
Cf.~\cite[proof of Lemma 4.1${}^\op$]{brown}.
\end{proof}
\end{lemma}

Accordingly, any such functor $F\colon\mathscr C\to\mathscr D$ descends to a functor $\Ho(F)\colon\Ho({\mathscr C})\to\Ho({\mathscr D})$.

\begin{defi}\label{defi:we-cof-cat}
An exact functor $F\colon\mathscr C\to\mathscr D$ is called a \emph{weak equivalence} if $\Ho(F)$ is an equivalence of categories.
\end{defi}

While this may sound like a na\"\i{}ve definition at first (especially from the viewpoint of higher category theory), Cisinski has shown that in a precise sense such a map ``preserves all quasi-categorical data'' (cf.~\cite[Théorème 3.19]{approximation}, which we generalize in Corollary~\ref{cor:we-diag}, and also \cite[Théorème 3.25$^\op$]{we-good}).

We have the following meta-result:

\begin{thm}[Szumi\l{}o]
The large category $\cat{COFCAT}$ of cofibration categories carries the structure of a fibration category with weak equivalences as defined above.
\begin{proof}
Cf.~\cite[Theorem 2.8]{szumilo-cofcat}
\end{proof}
\end{thm}

We will never need information about the fibrations of this structure; however, the definition can be found as \cite[Definition 2.3]{szumilo-cofcat}.

Finally we turn our attention to diagrams in ABC cofibration categories. We recall that in model categories, though suitable model structures on diagram categories need not always exist, they do exist whenever the indexing category is a so-called \emph{Reedy category}. ABC cofibration categories allow a one-sided variant of this:

\begin{defi}
A \emph{direct category} is a category $I$ such that there exists a functor $\deg\colon I\to\textbf{N}$ into the poset of natural numbers with the following property: any morphism $f\colon X\to Y$ such that $\deg X=\deg Y$ is an identity morphism (in particular $X=Y$).
\end{defi}

In other words: there is a ``degree function'' on the objects such that no morphism lowers degree and every non-identity morphism actually strictly raises degree. One should think of this degree function as a tool for allowing inductive constructions and proofs.

\begin{defi}
Let $I$ be a direct category and $i\in I$ an object. The \emph{$i$-th latching category} ${\mathscr L}_iI$ is the full subcategory of the slice category $I\downarrow i$ on all objects except $\id_i$.

Let now $\mathscr C$ be an arbitrary category and $X\colon I\to\mathscr C$ a functor. The \emph{$i$-th latching object} $L_iX$ is (if it exists) the colimit of the composition
\begin{equation*}
\begin{tikzcd}
\mathscr{L}_iI\arrow[r,"\text{forget}\,"] &[.5em] I\arrow[r, "X"] &[-.5em] \mathscr C.
\end{tikzcd}
\end{equation*}
The object $\id_i$ of $I\downarrow i$ is terminal and thus the inclusion $\mathscr{L}_iI\hookrightarrow I\downarrow i$ induces a natural map $L_iX\to X_i$, called the \emph{$i$-th latching map}.
\end{defi}

\begin{defi}
Let $I$ be a (not necessarily small) direct category and $\mathscr C$ be an ABC cofibration category. Then a diagram $X\colon I\to\mathscr C$ is called \emph{Reedy cofibrant} if for every $i\in I$ the latching object $L_iX$ exists, is cofibrant, and the latching map $L_iX\to X_i$ is a cofibration.

A map $f\colon X\to Y$ with $X$ and $Y$ Reedy cofibrant is called a \emph{Reedy cofibration} if for all $i\in I$ the induced map $X_i\amalg_{L_iX}L_iY\to Y_i$ is a cofibration.
\end{defi}

\begin{prop}
Let $\mathscr C$ be an ABC cofibration category and $I$ a small direct category. Then the Reedy cofibrations together with the levelwise weak equivalences turn ${\mathscr C}^I$ into an ABC cofibration category. In particular, the subcategory ${\mathscr C}_{\text{R}}^I$ of Reedy cofibrant diagrams becomes a cofibration category.

Moreover, if $I$ is in addition a category with weak equivalences, then this restricts to the structure of an ABC cofibration category on the full subcategory of homotopical diagrams and to the structure of a cofibration category on the full subcategory of homotopical Reedy cofibrant diagrams.
\begin{proof}
Cf.~\cite[Theorem 9.2.4-(1a)]{cof} and \cite[Theorem 9.3.8-(1a)]{cof}.
\end{proof}
\end{prop}

The above says in particular that for any homotopical diagram $X\colon I\to\mathscr C$ we can find a Reedy cofibrant homotopical diagram $\widehat X\colon I\to\mathscr C$ together with a weak equivalence $\widehat{X}\to X$. The following relative version of this result will be extremely useful:

\begin{defi}
A \emph{sieve} is a fully faithful embedding $f\colon I\hookrightarrow J$ such that the following holds: if $j\in J$ is an object such that there exists a morphism $j\to f(i)$ for some $i\in I$ then $j$ lies in the  image of $f$.
\end{defi}

We emphasize that in the above defintion we are considering the honest image (as opposed to the essential image).

\begin{lemma}\label{lemma:relative-cofibrant-replacement}
Let $I\hookrightarrow J$ be a homotopical sieve of small homoptical direct categories (i.e.~a homotopical functor that is at the same time a sieve between the underlying categories) and let $\mathscr C$ be a cofibration category. Assume we are given a homotopical diagram $X\colon J\to {\mathscr C}$ and a weak equivalence $f\colon H\to X|_I$ such that $H$ is Reedy cofibrant and homotopical. Then there exists a Reedy cofibrant homotopical diagram $\widehat X$ together with a weak equivalence $g\colon\widehat X\to X$ such that $\widehat{X}|_I=H$ and $g|_I=f$. 
\begin{proof}
This is a special case of \cite[Lemma 1.9-(1)]{szumilo-frames}.
\end{proof}
\end{lemma}

We will without further mention use the following:

\begin{lemma}
Let $I,J$ be direct categories with weak equivalences. Then the exponential law isomorphism restricts to an isomorphism ${\mathscr C}^{I\times J}_{\text{R}}\cong({\mathscr C}^I_{\text{R}})^J_{\text{R}}$ of cofibration categories.
\begin{proof}
It is obvious that a diagram $I\times J\to\mathscr C$ is homotopical if and only if its adjunct $J\to\mathscr C^I$ is, and analogously for weak equivalences. For the corresponding statements regarding Reedy cofibrations cf.~\cite[proof of Lemma 3.14]{frame-loc}.
\end{proof}
\end{lemma}

The axioms of an ABC cofibration category guarantee the existence of some specific colimits. The following result greatly extends this:

\begin{prop}\label{prop:reedy-colim}
Let $\mathscr C$ be an ABC cofibration category, $I$ a small direct category, and $X\colon I\to\mathscr C$ Reedy cofibrant. Then $\colim_I X$ exists and is cofibrant. Moreover, this colimit is preserved by any exact functor.
\begin{proof}
The first statement is \cite[Theorem 9.3.5-(1a)]{cof} and the second one follows immediately from the explicit construction given there.
\end{proof}
\end{prop}

With this established, it easily follows:

\begin{prop}
Let $f\colon\mathscr C\to\mathscr D$ be an exact functor of cofibration categories and $I$ a small direct category. Then pushforward along $f$ yields an exact functor $\mathscr C^I_{\text{R}}\to\mathscr D^I_{\text{R}}$, in particular it preserves Reedy cofibrant diagrams.\qed
\end{prop}

What makes ABC cofibration categories very convenient is that they are in full generality closed under forming diagram categories:

\begin{thm}[Cisinski, R{\u a}dulescu-Banu]\label{thm:abc-diagram}
Let $\mathscr C$ be an ABC cofibration category and $I$ a small category. Then ${\mathscr C}^I$ equipped with the levelwise weak equivalences and levelwise cofibrations is an ABC cofibration category. If $I$ is a category with weak equivalences, the same holds true for the full subcategory of homotopical diagrams.

In particular, for every cofibration category $\mathscr C$ and every small category $I$, ${\mathscr C}^I$ equipped with the levelwise weak equivalences and cofibrations is again a cofibration category. If $I$ is a category with weak equivalences, the same holds true for the full subcategory of homotopical diagrams.
\begin{proof}
Cf.~\cite[Theorem 9.5.5-(1)]{cof} and \cite[Theorem 9.5.6-(1)]{cof}.
\end{proof}
\end{thm}

We remark that in the above generality this is a deep theorem that relies heavily on the existence of infinite colimits for constructing the factorizations. It is a triviality if one assumes functorial factorizations.

\begin{rk}
Let $I, J$ be categories with weak equivalences and let $\mathscr C$ be a cofibration category. Then the exponential law isomorphism restricts to an isomorphism ${\mathscr C}^{I\times J}\cong({\mathscr C}^I)^J$ of cofibration categories.
\end{rk}

We conclude with some easy observations about the functoriality of this construction:

\begin{lemma}
Let $I,J$ be categories with weak equivalences, and let $\mathscr C,\mathscr D$ be cofibration categories.
\begin{enumerate}
\item If $f\colon I\to J$ is a homotopical functor, then restriction along $f$ yields an exact functor $f^*\colon\mathscr C^J\to\mathscr C^I$.
\item If $F\colon \mathscr C\to\mathscr D$ is an exact functor, then pushforward along $F$ yields an exact functor $F_*\colon\mathscr C^I\to\mathscr D^I$.\qed
\end{enumerate}
\end{lemma}

\subsection{Derivators}
The theory of \emph{derivators} was invented by Grothendieck in \cite{derivateurs} and then further studied among others by Cisinski, Maltsiniotis, Keller, and Groth, cf.~e.g.~\cite{basic-loc,modelgood,maltsi,groth-publ}. A slight variation of this notion was also independently introduced and studied by Heller \cite{heller}.

We are mostly following the monograph \cite{groth} and the article \cite{groth-publ} here (in particular our convention on $2$-cells is ``opposite'' to the one used by Grothendieck and Cisinski).

\begin{defi}\label{defi:preder}
A \emph{prederivator} $\mathscr D$ is a strict $2$-functor $\cat{Cat}^\op\to\cat{CAT}$. We call $\mathscr D(*)$ the \emph{underlying category} of $\mathscr D$.
\end{defi}

\begin{ex}
Let $\mathscr C$ be a category. Then $\Yo(\mathscr C)\mathrel{:=}\cat{CAT}(\blank,{\mathscr C})$ is a prederivator, called the \emph{prederivator represented by $\mathscr C$}. Its underlying category is canonically isomorphic to $\mathscr C$ (by evaluating at the unique object of $*$).
\end{ex}

\begin{ex}
Let $\mathscr D$ be a prederivator, $A$ a small category. Then $\mathscr D^A\mathrel{:=}\mathscr D(A\times\blank)$ is again a prederivator.
\end{ex}

\begin{constr}
Let $\mathscr D$ be prederivator and $I$ a small category. Then we have for any $i\in I$ an evaluation functor $\text{ev}_i\mathrel{:=} \text{incl}_i^*\colon \mathscr D(I)\to\mathscr D(*)$ and moreover, for any morphism $f\colon i\to j$ in $I$ a natural transformation $\incl_i\Rightarrow\incl_j$ inducing $\ev_i\Rightarrow\ev_j$. Together these assemble into a functor $\diag\colon\mathscr D(I)\to\mathscr D(*)^I$ called the \emph{underlying (incoherent) diagram functor}.

Similarly, if $A$ is another small category, we have a \emph{partial underlying diagram functor} $\mathscr D(A\times I)\to\mathscr D(A)^I$.
\end{constr}

Note that while this map is an isomorphism in the case of represented prederivators it will in general not even be an equivalence.

\begin{defi}
The (strict) $2$-category of prederivators has
\begin{enumerate}
\item \emph{objects} the prederivators
\item \emph{morphisms} the pseudonatural transformations
\item \emph{2-cells} the modifications. 
\end{enumerate}
\end{defi}

The following is a well-known result in $2$-category theory; for example, a very similar result has been proven by Kelly, cf.~\cite[Proposition 1.3]{kelly}. However, we couldn't find an explicit proof of the version stated below, so we briefly sketch the (classical) argument:

\begin{lemma}\label{lemma:der-quasi-inverse}
Let $F,G\colon\mathscr C\to\mathscr D$ be strict $2$-functors of strict $2$-categories and let $\sigma\colon F\Rightarrow G$ be a pseudonatural transformation, such that for each $c\in\mathscr C$ the morphism $\sigma_c$ is an equivalence in $\mathscr D$. Then $\sigma$ is an equivalence in the $2$-category of strict $2$-functors $\mathscr C\to\mathscr D$ (a \emph{pseudonatural equivalence}), i.e.~there exists a pseudonatural transformation $\tau\colon G\Rightarrow F$ such that $\sigma\tau\cong\id$ and $\tau\sigma\cong\id$. Moreover, any choice of levelwise quasi-inverses of $\sigma$ gives rise to such a quasi-inverse $\tau$.
\begin{proof}[Sketch of proof]
We fix for each $I\in\mathscr C$ a quasi-inverse $\tau_I\colon G(I)\to F(I)$ of $\sigma_I$ and (invertible) $2$-cells $\epsilon_I\colon \tau_I\sigma_I\Rightarrow\id$ and $\eta_I\colon \id\Rightarrow\sigma_I\tau_I$ exhibiting the pair $\tau_I,\sigma_I$ as an adjoint equivalence in $\mathscr D$.

Denote by $\gamma_u$ the structure isomorphisms of $\sigma$. We then define for any morphism $u\colon I\to J$ in $\mathscr C$ the $2$-cell $\gamma_u^\prime$ as the inverse of the pasting
\begin{equation*}
\begin{tikzcd}
G(I) \arrow[dr, "\id"'{name=a}, bend right=20pt] \arrow[r, "\tau_I"] & F(I) \arrow[d,"\sigma_I"] \twocell[from=a, "\scriptstyle\eta_I"{shift={(-7pt,6pt)}}] \arrow[r, "F(u)"] & F(J) \arrow[d, "\sigma_J"']\arrow[dr, "\id"{name=b}, bend left=20pt]\\
 & G(I) \twocell[ur, "\scriptstyle\gamma_u"{shift={(-7pt,6pt)}}] \arrow[r, "G(u)"'] & G(J) \twocell[to=b, "\scriptstyle\epsilon_I"{shift={(-7pt,6pt)}}]\arrow[r, "\tau_J"'] & F(J)
\end{tikzcd}
\end{equation*}
(we remark that the above is indeed invertible as a pasting of invertible cells), i.e.~the inverse of the canonical mate of $\gamma_u$ with respect to the above adjunction.

A straight-forward albeit lengthy calculation shows that this makes $\tau$ into a pseudonatural transformation and moreover that the $\eta_I$, $\epsilon_I$ assemble into invertible modifications $\id\Rrightarrow\sigma\tau$ and $\tau\sigma\Rrightarrow\id$ as desired.
\end{proof}
\end{lemma}

In particular a morphism $F\colon\mathscr D\to\mathscr E$ of prederivators is an equivalence if and only if it is so levelwise.

Prederivators on their own are not yet that interesting; what one actually studies is the following:

\begin{defi}\label{defi:derivator}
A prederivator $\mathscr D$ is called a \emph{right derivator} (resp.~\emph{left derivator}) if it satisfies the following axioms:
\begin{enumerate}
\item If $I$ is a set and $(A_i)_{i\in I}$ is a family of small categories, then the map
\begin{equation*}
(\text{ev}_i)_{i\in I}\colon\mathscr D\left(\coprod_{i\in I} A_i\right)\to\prod_{i\in I}{\mathscr D}(A_i)
\end{equation*}
is an equivalence of categories. In particular ${\mathscr D}(\varnothing)$ is equivalent to the terminal category $*$.
\item For any small category $A$ the functor $\diag\colon{\mathscr D}(A)\to{\mathscr D}(*)^A$ is conservative.
\item (\emph{Existence of homotopy Kan extensions}) For each functor $u\colon A\to B$ of small categories the restriction $u^*\colon{\mathscr D}(B)\to{\mathscr D}(A)$ has a left (resp.~right) adjoint, which we denote by $u_!$  (resp.~$u_*$).
\item (\emph{Kan extensions are pointwise}) For any functor $u\colon A\to B$ of small categories and each $b\in B$ the canonical mate transformation of the slice square
\begin{equation*}
\begin{tikzcd}
\mathscr{D}(u\downarrow b) & \arrow[l, "\text{forget}^*"'] \mathscr{D}(A)\twocell[dl]\\
\mathscr{D}(*)\arrow[u, "\pr^*"] & \arrow[l,"\ev_b"] \mathscr{D}(B)\arrow[u, "u^*"']
\end{tikzcd}
\qquad\text{resp.}\qquad
\begin{tikzcd}
\mathscr{D}(b\downarrow u) & \arrow[l, "\text{forget}^*"'] \mathscr{D}(A)\\
\mathscr{D}(*)\twocell[ur]\arrow[u, "\pr^*"] & \arrow[l,"\ev_b"] \mathscr{D}(B)\arrow[u, "u^*"']
\end{tikzcd}
\end{equation*}
(which is a natural transformation $\pr_!\circ\text{forget}^*\Rightarrow \ev_b\circ u_!$ resp.~$\ev_b\circ u_*\Rightarrow \pr_*\circ\text{forget}^*$) is an isomorphism.
\end{enumerate}

Finally, $\mathscr D$ is called a \emph{derivator} if it is both a left and a right derivator.
\end{defi}

\begin{warn}
The choice of ``left'' and ``right'' above is the original one made by Grothendieck and later adopted by Cisinski. It refers to the terminology ``left exact'' and ``right exact'' rather than ``left adjoint'' and ``right adjoint,'' so in a \emph{right} derivator the restriction functors have \emph{left} adjoints.

We remark however, that also the opposite convention is in use by some authors following Heller.
\end{warn}

\begin{ex}\label{ex:repr-derivator}
Let $\mathscr C$ be a complete and cocomplete category. Then $\Yo(\mathscr C)$ is a derivator which we accordingly call the \emph{derivator represented by $\mathscr C$}. Indeed the above should be seen as a very basic axiomatization of the classical theory of Kan extensions.
\end{ex}

A straight-forward albeit lengthy calculation shows:

\begin{lemma}
Let $F\colon\mathscr D\to\mathscr E$ be an equivalence of prederivators. Then $\mathscr D$ is a left derivator (resp.~right derivator resp.~derivator) if and only if $\mathscr E$ is.\qed
\end{lemma}

While Example~\ref{ex:repr-derivator} provides a first sanity check for the axioms in Definition~\ref{defi:derivator}, this does not yet connect derivators to abstract homotopy theory.

\begin{defi}\label{defi:homotopy-der}
Let $\mathscr C$ be a category with weak equivalences. Then its \emph{homotopy prederivator} $\HOrel(\mathscr C)$ is defined as the composition of strict $2$-functors
\begin{equation*}
\cat{Cat}^\op\to\cat{CATWE}\stackrel{\Ho}{\longrightarrow}\cat{CAT},
\end{equation*}
where the first map is the obvious lift of $\Yo(\mathscr C)$, cf.~Example~\ref{ex:levelwise-we}. We remark that the second map is indeed a strict $2$-functor since we choose \emph{strict} localizations.
\end{defi}

We make $\HOrel$ into a strict $1$-functor $\cat{CATWE}\to\cat{PREDER}$ via pushforward.

In general, $\HOrel$ will be far from a derivator (and it is also not the correct notion, cf.~Example~\ref{ex:not-strong} below). However, we have the following result:

\begin{thm}[Cisinski]\label{thm:modelgood}
Let $\mathscr C$ be a model category. Then the homotopy pre\-derivator of its underlying category with weak equivalences is a derivator.
\begin{proof}
Cf.~\cite[Théorème 6.11]{modelgood}.
\end{proof}
\end{thm}

\begin{rk}
Theorem~\ref{thm:modelgood} is a deep result and even establishing the existence of general homotopy (co)limit functors is quite some work. The main reason for this is, that in general we have no canonical model structures on ${\mathscr C}^I$ for a model category $\mathscr C$ and a small category $I$. In the case of \emph{combinatorial model categories} where such a structure exists, a much simpler proof can be given, cf.~\cite[Proposition 1.36]{groth-publ}.
\end{rk}

We also have the following ``one-sided version'' of Theorem~\ref{thm:modelgood}:

\begin{thm}[Cisinski]\label{thm:cof-hoder}
Let $\mathscr C$ be an ABC cofibration category. Then the homotopy prederivator of its underlying category with weak equivalences is a right derivator.
\begin{proof}
Cf.~\cite[Corollaire 6.21${}^\op$]{approximation}.
\end{proof}
\end{thm}

We will denote this right derivator by $\HOcof({\mathscr C})$. While we will be mostly be interested in the case of an actual cofibration category, we note the following:

\begin{cor}\label{cor:cof-approx}
Let $\mathscr{C}$ be an ABC cofibration category. Then the inclusion $\mathscr{C}_c\hookrightarrow\mathscr{C}$ induces an equivalence $\HOcof(\mathscr{C}_c)\to\HOcof(\mathscr{C})$.
\begin{proof}
Let $I$ be a small category. Then the levelwise weak equivalences and cofibrations make $\mathscr{C}^I$ into a cofibration category by Theorem~\ref{thm:abc-diagram}. In particular, the inclusion $(\mathscr{C}_c)^I=(\mathscr{C}^I)_c\hookrightarrow\mathscr{C}$ descends to an equivalence of homotopy categories by Proposition~\ref{prop:cof-approx-basis}; the claim follows.
\end{proof}
\end{cor}

Finally we will use the following construction:

\begin{defi}
Let $\mathscr C$ be a quasi-category. We write $\HO_\infty({\mathscr C})$ for the prederivator given by $\HO_\infty({\mathscr C})(I)=\h({\mathscr C}^{\nerve I})$ and similarly on functors and natural transformations.
\end{defi}

Here $\h$ denotes the (unenriched) homotopy category of a given quasi-category or simplicial set, cf.~\cite[Proposition 1.2.3.1]{htt}.
Again we can extend this to a strict $1$-functor $\cat{QCAT}\to\cat{PREDER}$ in the obvious way.

In Corollary~\ref{cor:quasi-derivator} we will in particular show that this prederivator is a derivator if $\mathscr C$ is complete and cocomplete.

\begin{prop}[Joyal]\label{prop:hoinf-we}
Let $f\colon\mathscr C\to\mathscr D$ be a weak equivalence of quasi-categories. Then the induced map $\HOinf(f)\colon\HOinf(\mathscr C)\to\HOinf(\mathscr D)$ is an equivalence.
\begin{proof}
This is immediate from \cite[Proposition 1.2.7.3-(2)]{htt}.
\end{proof}
\end{prop}

While the basic idea of derivator theory is, that one should exclusively think about coherent diagrams, for comparison with older approaches (e.g.~triangulated categories) it is sometimes necessary to deal with some incoherent diagrams, e.g.~morphisms in the underlying category. To apply the theory of derivators to them, we have to \emph{lift} them to coherent diagrams. This motivates the following conditions:

\begin{lud}
A small category $F$ is called \emph{free} if the following equivalent conditions are satisfied:
\begin{enumerate}
\item There exists a quiver $Q$ such that $F\cong\mathop{\text{free}}Q$ for some hence any left adjoint $\text{free}$ of the forgetful functor $\cat{Cat}\to\cat{Quivers}$.
\item There are sets $I,J$ and a pushout square
\begin{equation*}
\begin{tikzcd}
\coprod_{i\in I}\partial[1]\arrow[dr, phantom, "\ulcorner", very near end] \arrow[d] \arrow[r, "\coprod\text{incl}"] & \coprod_{i\in I}[1]\arrow[d]\\
\coprod_{j\in J}[0] \arrow[r] & F
\end{tikzcd}
\end{equation*}
in $\cat{Cat}$. Here $\partial[1]$ denotes the subcategory of $[1]$ where we remove the only non-identity morphism (i.e.~it is the discrete category with two objects).
\item There exists a $1$-skeletal simplicial set $K$ such that $F\cong\h K$.\qed
\end{enumerate}
\end{lud}

\begin{defi}
A prederivator $\mathscr D$ is called \emph{strong} if the partial underlying diagram functor $\mathscr D(A\times F)\to \mathscr D(A)^F$ is full and essentially surjective for every small category $A$ and every small free category $F$.
\end{defi}

\begin{warn}
There are various definitions of ``strong'' in the literature and the above one is a bit more restrictive than the usual ones: normally one requires the above only for \emph{finite} free categories or sometimes even only for $F=[1]$. In interesting cases the underlying diagram functor is \emph{never} faithful (except for discrete categories). Explicit examples for how this fails can be found e.g.~in \cite[Section 7.5]{groth}.
\end{warn}

Again an easy calculation shows:

\begin{lemma}
Let $\mathscr D\to\mathscr E$ be an equivalence of prederivators. Then $\mathscr D$ is strong if and only if $\mathscr E$ is.\qed
\end{lemma}

\begin{ex}
We will show that almost all homotopy prederivators discussed above are strong: namely for the case of quasi-categories this is Proposition~\ref{prop:hoinf-strong}, and for ABC cofibration categories (which includes both model and cofibration categories) we prove this as Corollary~\ref{cor:hocof-strong}.
\end{ex}

\begin{ex}\label{ex:not-strong}
The homotopy prederivator of a general category with weak equivalences need not be strong (even in the weakest sense): as an example, let $\mathscr C$ be
\begin{equation*}
A\longrightarrow B\buildrel\sim\over\longleftarrow C\longrightarrow D
\end{equation*}
(which is even homotopical). Then $\Ho\mathscr C$ contains a morphism $A\to D$ but this is not in the essential image of $\Ho({\mathscr C}^{[1]})\to(\Ho\mathscr C)^{[1]}$: namely, a (non-strict) localization of $\mathscr C$ is given by collapsing the arrow $C\to B$; hence neither $A$ nor $D$ are isomorphic to any other object of $\Ho(\mathscr C)$. Since there is no arrow $A\to D$ in $\mathscr C$ the claim follows.

We remark that in view of Proposition~\ref{prop:hoinf-strong} this tells us that $\HOrel(\mathscr{C})$ does not arise as the homotopy prederivator of \emph{any} quasi-category. In particular, the lower right triangle of the diagram in the introduction does \emph{not} commute up to equivalence.
\end{ex}

\begin{ex}
There are also naturally arising derivators that are not strong, even with respect to the weakest definition mentioned above, cf.~\cite[Example 5.5]{non-strong}.
\end{ex}

\section{Szumi\l{}o's quasi-category of frames}
We recall Szumilo's definition of the \emph{quasi-category of frames} $\nerve_f(\mathscr{C})$ associated to a cofibration category $\mathscr C$ along with the most important results as developed in \cite{szumilo} and then subsequently published as \cite{szumilo-frames, szumilo-cofcat, szumilo-equi}.

\begin{constr}\label{constr:d}
Let $K$ be a (possibly large) simplicial set. We define a homotopical category $DK$ (the ``thick barycentric subdivision of $K$'') as follows: objects of $DK$ are pairs $(n,\sigma)$ where $n\in\mathbb{N}$ and $\sigma\in K_n$ is an $n$-simplex. A morphism $(m,\sigma)\to(n,\tau)$ is an injective monotone map $i\colon[m]\to[n]$ such that $i^*\tau=\sigma$.

We have a simplicial map $p\colon\nerve(DK)\to K$ as follows: an $n$-simplex 
\begin{equation*}
(\sigma_0,k_0)\xrightarrow{i_0}\cdots\xrightarrow{i_{n-1}}(\sigma_n,k_n)
\end{equation*}
of $\nerve(DK)$ is sent to $f^*\sigma_n$ where $f\colon [n]\to[k_n]$ is the (not necessarily injective) monotone map given by $f(j)=(i_{n-1}\cdots i_{j})(n_j)$; in particular we have $f(n)=k_n$. A straightforward calculation shows that this is indeed a simplicial map.

We now take the weak equivalences in $DK$ to be the smallest class of maps closed under $2$-out-of-$6$ and containing all the maps that are sent by $p$ to degenerate edges.

We observe that $D$ becomes a functor $\cat{SSET}\to\cat{HCAT}$ via pushforward.
\end{constr}

We further remark that $DI$ is actually a direct category; a possible degree functor is given by $(n,\sigma)\mapsto n$. Moreover, its latching categories all always finite.

\begin{rk} In most situations $K$ will be the nerve of some small category $I$ in which case we simply write $DI\mathrel{:=}D(\nerve I)$. We note that as a category $DI$ is simply the slice $\Delta^\sharp\downarrow I$ where $\Delta^\sharp\subset\Delta$ is the subcategory of injective monotone maps, i.e.~objects of $DI$ are functors $[n]\to I$ for varying $n\in\mathbb{N}$ and a morphism from $X\colon[m]\to I$ to $Y\colon[n]\to I$ is an injective monotone map $i\colon[m]\to [n]$ such that $X=Y\circ i$.

By full faithfulness of the nerve, $p\colon\nerve DI\to \nerve I$ is induced by a functor $DI\to I$ which we denote by $p$ again. This functor can be described explicitly as follows: an object $X\colon[m]\to I$ is sent to $X(m)$ and a morphism from $X$ to $Y\colon[n]\to I$ is sent to $Y(i(m)\to n)$.

A morphism in $DI$ is a weak equivalence if and only if its image under $p$ is an isomorphism, cf.~\cite[Lemma 2.4]{szumilo-frames}.
\end{rk}

\begin{lemma}\label{lemma:representability}
\begin{enumerate}
\item As a functor $\cat{SSET}\to\cat{CAT}$ (i.e.~disregarding the homotopical structure) $D$ preserves ($\textbf{V}$-small) colimits.
\item For any category $I$, any functor $K_\bullet\colon I\to\cat{SSET}$, and any cofibration category $\mathscr C$ the induced bijection
\begin{equation*}
\Hom_{\cat{CAT}}\big(D(\colim\nolimits_I K_\bullet), \mathscr C\big)\to\lim\nolimits_I\Hom_{\cat{CAT}}\big(D(K_\bullet),\mathscr C\big)
\end{equation*}
restricts to bijections between the corresponding subsets of Reedy cofibrant respectively homotopical diagrams.
\end{enumerate}
\begin{proof}
The first statement is \cite[Lemma 2.5]{szumilo-frames} and the second one is \cite[proof of Proposition 2.6]{szumilo-frames}.
\end{proof}
\end{lemma}

The following will be very useful, in particular in conjunction with Lemma~\ref{lemma:relative-cofibrant-replacement}:

\begin{lemma} Let $f\colon K\to L$ be a map of (possibly large) simplicial sets. Then $(Df)^*\colon\mathscr{C}^{DL}\to\mathscr{C}^{DK}$ preserves Reedy cofibrations. Moreover, if $f$ is injective, then $Df$ is a sieve.
\begin{proof}
For the first statement cf.~\cite[proofs of Lemma 2.1 and Proposition 2.6]{szumilo-frames}; the second statement is immediate from inspection.
\end{proof} 
\end{lemma}

Together these results imply:

\begin{dac}[Szumilo]
Let $\mathscr C$ be a cofibration category. Then the functor $\cat{SSET}\to\cat{SET}$ given by the assignment
\begin{equation}\label{eq:functor}
K\mapsto\{\text{Reedy cofibrant homotopical diagrams $DK\to\mathscr C$}\}
\end{equation}
together with the obvious restriction maps is representable by a large simplicial set, given explicitely by
\begin{equation*}
(\nerve_f\mathscr C)_n=\{\text{Reedy cofibrant homotopical diagrams $D[n]\to\mathscr C$}\}.
\end{equation*}
$\nerve_f(\mathscr C)$ is called the \emph{quasi-category of frames} of $\mathscr C$.
\begin{proof}
Since $\cat{SSET}$ is a (large) presheaf category, it suffices to show that (\ref{eq:functor}) sends (not necessarily small) colimits to limits. This is immediate from Lemma~\ref{lemma:representability}.
\end{proof}
\end{dac}

Since Reedy cofibrant and homotopical diagrams are stable under pushforward, $\nerve_f$ becomes a functor into  $\cat{SSET}$. The following says in particular that the above name is not ill-chosen:

\begin{thm}[Szumi\l{}o]
The functor $\nerve_f$ takes values in cocomplete quasi-categories and cocontinuous functors.
\begin{proof}
Cf.~\cite[Theorem 3.1]{szumilo-equi}.
\end{proof}
\end{thm}

We can also describe the equivalences in this quasi-category; for this we will use the following variant of Construction~\ref{constr:d}:

\begin{constr}
Let $I$ be a homotopical category and denote by $\forget I$ its underlying category. We define $DI$ to be the following homotopical category: as a category we take $DI=D(\forget I)$ and the weak equivalences are created by $p\colon DI\to I$.
\end{constr}

\begin{lemma}\label{lemma:eq-in-nf}
A morphism in $\nerve_f(\mathscr{C})$ given by $f\colon D[1]\to\mathscr{C}$  is an equivalence if and only if it is homotopical when regarded as $D\widehat{[1]}\to\mathscr{C}$.
\begin{proof}
Cf.~\cite[Corollary 3.7]{szumilo-frames}.
\end{proof}
\end{lemma}

\begin{thm}[Szumi\l{}o]\label{thm:frame-equiv}
The Joyal model structure restricts to the structure of a fibration category on the subcategory $\cat{QCAT}_!$ of cocomplete quasi-categories with cocontinuous functors as morphisms.

Moreover, the functor
\begin{equation*}
\nerve_f\colon\cat{COFCAT}\to\cat{QCAT}_!
\end{equation*}
is a weak equivalence of fibration categories (in particular exact).
\begin{proof}
cf.~\cite[Theorem 4.9]{szumilo-equi}.
\end{proof}
\end{thm}

We finish with some statements about diagrams in cofibration categories and their associated quasi-categories.

\begin{prop}[Szumi\l{}o \& Kapulkin]\label{prop:pr-equiv}
Let $K, L$ be simplicial sets, $\mathscr C$ a cofibration category. Then $(D\pr_1,D\pr_2)\colon D(K\times L)\to DK\times DL$ induces a well-defined and exact functor ${\mathscr C}^{DK\times DL}_{\text{R}}\to \mathscr C^{D(K\times L)}_{\text{R}}$ and this is a weak equivalence.
\begin{proof}
The case $K=\Delta^m$, $L=\Delta^n$ appears as \cite[Proposition 4.5]{frame-loc}. The general case follows now from \cite[proof of Lemma 4.4]{frame-loc}.
\end{proof}
\end{prop}

\begin{constr}\label{constr:frame-diag}
Let $\mathscr C$ be a cofibration category and let $K$ be a simplicial set. We will construct a simplicial map
\begin{equation*}
\Phi\colon \nerve_f\left({\mathscr C}^{DK}_{\text{R}}\right)\to\nerve_f({\mathscr C})^K
\end{equation*}
as follows: on $n$-simplices $\Phi$ is given by
\begin{align*}
\nerve_f\left({\mathscr C}^{DK}_{\text{R}}\right)_n ={}& R\left(D[n], {\mathscr C}^{DK}_{\text{R}}\right)\\
\cong{}& R(DK\times D[n], {\mathscr C})\\
\to{}& R(D(K\times[n]), {\mathscr C})\\
\cong{}& \Hom\big(K\times \Delta^n, \nerve_f({\mathscr C})\big)\\
={}& \left(\nerve_f({\mathscr C})^K\right)_n.
\end{align*}
Here $R(J,{\mathscr D})$ denotes the set of Reedy cofibrant homotopical diagrams $J\to\mathscr D$, the isomorphisms come from the definition of $\nerve_f$ respectively the obvious adjunction and the remaining arrow is restriction along $D(K\times[n])\to DK\times D[n]$.
\end{constr}

\begin{thm}[Szumi\l{}o \& Kapulkin]\label{prop:phi-equiv}
The map from Construction~\ref{constr:frame-diag} is well-defined and an equivalence.
\begin{proof}
Cf.~\cite[Corollary 4.16]{frame-loc}.
\end{proof}
\end{thm}

\section{The comparison result}
\subsection{Homotopy right derivators of cofibration categories}
We recall from Theorem~\ref{thm:cof-hoder} that for any cofibration category $\mathscr C$ its homotopy prederivator $\HOcof(\mathscr C)$ is a right derivator. We begin by introducing a ``thickened'' variant of this for which we will need:

\begin{prop}\label{prop:cof-hoder-models}
Let $\mathscr C$ be a cofibration category and $I$ a small category.
\begin{enumerate}
\item The inclusion ${\mathscr C}^{DI}_{\text{R}}\hookrightarrow{\mathscr C}^{DI}$ is a weak equivalence.\label{item:reedy-repl}
\item The map $p^*\colon{\mathscr C}^I\to{\mathscr C}^{DI}$ (recall Construction~\ref{constr:d}) is a weak equivalence.\label{item:p-star}
\end{enumerate}
\begin{proof}
The first statement is a special case of \cite[Proposition 1.7-(3)]{szumilo-frames}. The second map is trivially exact and it descends to an equivalence of homotopy categories by \cite[Theorem 9.5.8-(1)]{cof}, or alternatively \cite[Théorème 6.17${}^\op$ and Lemme 2.5${}^\op$]{approximation}.
\end{proof}
\end{prop}

\begin{cor}\label{cor:incl-d}
Let $n\ge 0$. Then the map $i\colon [n]\to D[n]$, that sends $a$ to the inclusion $[a]\hookrightarrow[n]$, induces a quasi-inverse to $p^*\colon\mathscr C^{I}\to\mathscr C^{DI}$. In particular, $i^*$ descends to an equivalence on homotopy categories.
\begin{proof}
By functoriality $i^*p^*=\id$ and the claim follows from Proposition~\ref{prop:cof-hoder-models}-(\ref{item:p-star}).
\end{proof}
\end{cor}

We remark that it is also possible to prove the corollary directly by elementary means, cf.~\cite[proof of Lemma~3.2]{szumilo-frames}.

\begin{defi}
Let $\mathscr C$ be a cofibration category. Then we have a $1$-functor
\begin{equation*}
\HOD(\mathscr C)\colon\cat{Cat}^\op \to \cat{CAT}
\end{equation*}
given by $\HOD(\mathscr C)(I)=\Ho({\mathscr C}^{DI})$ together with the obvious restrictions.
\end{defi}

We will now extend this to a prederivator.

\begin{rk}
In the following we will often appeal to the calculus of mates as a coherent way of inverting equivalences. In some cases it will be convenient to not keep track of the direction of various natural isomorphisms; since we always mention the respective adjunctions explicitly, no ambiguity will arise from this.
\end{rk}

\begin{constr}\label{constr:homotopic-eval}
Let $\mathscr C$ be a cofibration category and $I$ a small category. We have a commutative diagram
\begin{equation*}
\begin{tikzcd}[column sep=.75in]
\Ho\big(\mathscr{C}^{DI\times D[0]}\big)\arrow[d, "{(D\pr_1,D\pr_2)^*}"'] & \Ho\big(\mathscr{C}^{DI}\big)\arrow[l, "\pr_1^*"'] \arrow[d,"\id"]\\
\Ho\big(\mathscr{C}^{D(I\times[0])}\big) & \Ho\big(\mathscr{C}^{DI}\big).\arrow[l, "D(\pr_1)^*"]
\end{tikzcd}
\end{equation*}
An honest inverse of the lower map is given by $D(\id,0)^*$ and a quasi-inverse to the top map is given by $(\id,[0]\to[0])^*$. Here and in what follows we denote a functor $A\to B$ constant at some object $b\in B$ simply by $b$; in other words, the chosen quasi-inverse of the top map sends a diagram $X\colon DI\times D[0]\to\mathscr{C}$ to the diagram $\widetilde X\colon DI\to\mathscr{C}$ with $\widetilde X(f\colon[n]\to I)=X(f,[0]\to[0])$.

We upgrade these two pairs to adjoint equivalences with the original maps as right adjoints in such a way that for the lower adjunction both unit and counit are the identity and in the case of the upper adjunction the counit is the identity. We pass to canonical mates with respect to these adjunctions, yielding
\begin{equation}\label{diag:h-e-mate}
\begin{tikzcd}[column sep=.75in]
\Ho\big(\mathscr{C}^{DI\times D[0]}\big)\arrow[d, "{(D\pr_1,D\pr_2)^*}"']\arrow[r,"{(\id,[0]\to[0])^*}"] & \Ho\big(\mathscr{C}^{DI}\big)\arrow[d,"\id"]\\
\Ho\big(\mathscr{C}^{D(I\times[0])}\big)\arrow[r, "{D(\id,0)^*}"']\twocell[ur] & \Ho\big(\mathscr{C}^{DI}\big),
\end{tikzcd}
\end{equation}
and this natural transformation is an isomorphism because both adjunctions are adjoint equivalences. Now we have for any $j\colon [0]\to[1]$ a commutative diagram
\begin{equation}\label{diag:h-e-easy}
\begin{tikzcd}[column sep=.75in]
\Ho\big(\mathscr{C}^{DI\times D[1]}\big)\arrow[d, "{(D\pr_1,D\pr_2)^*}"']\arrow[r, "{(\id\times Dj)^*}"] & \Ho\big(\mathscr{C}^{DI\times D[0]}\big)\arrow[d,"{(D\pr_1,D\pr_2)^*}"]\\
\Ho\big(\mathscr{C}^{D(I\times[1])}\big)\arrow[r, "D(\id\times j)^*"'] & \Ho\big(\mathscr{C}^{D(I\times[0])}\big)
\end{tikzcd}
\end{equation}
and pasting with (\ref{diag:h-e-mate}) finally yields a natural isomorphism filling
\begin{equation*}
\begin{tikzcd}[column sep=1.25in]
\Ho\big(\mathscr{C}^{DI\times D[1]}\big)\arrow[d, "{(D\pr_1,D\pr_2)^*}"'] \arrow[r, "{\ev_j=(\id,j\colon [0]\to[1])^*}"] & \Ho\big(\mathscr{C}^{DI}\big)\arrow[d,"\id"]\\
\Ho\big(\mathscr{C}^{D(I\times[1])}\big)\arrow[r, "\ev_{j(0)}"'] & \Ho\big(\mathscr{C}^{DI}\big).
\end{tikzcd}
\end{equation*}
\end{constr}

\begin{constr}
Let $\mathscr C$ be a cofibration category, $f,g\colon I\to J$ functors of small categories, and let $\tau\colon f\Rightarrow g$ be a natural transformation between them, which we identify with a functor $t\colon I\times[1]\to J$. We will now construct a natural transformation $\tau_D^*\colon (Df)^*\Rightarrow (Dg)^*$ as follows: by functoriality we have a commutative diagram
\begin{equation}\label{diag:two-cells-top}
\begin{tikzcd}[column sep=large]
&\arrow[dl, "(Df)^*"', bend right=20pt] \Ho(\mathscr{C}^{DJ})\arrow[d, "(Dt)^*"] \arrow[dr, "(Dg)^*", bend left=20pt]\\
\Ho(\mathscr{C}^{DI}) & \arrow[l, "\ev_0"] \Ho(\mathscr{C}^{D(I\times[1])}) \arrow[r, "\ev_1"'] &\Ho(\mathscr{C}^{DI}).
\end{tikzcd}
\end{equation}
Moreover, the previous construction gives us a diagram
\begin{equation}\label{diag:two-cells-before}
\begin{tikzcd}[column sep=large]
\Ho(\mathscr C^{DI}) &\arrow[l, "\ev_0"'] \Ho(\mathscr C^{D(I\times[1])})\twocell[dl] \arrow[r, "\ev_1"] \twocell[dr]& \Ho(\mathscr C^{DI})\\
\Ho(\mathscr C^{DI})\arrow[u, "\id"] &\arrow[l,"\ev_{d_1}"] \Ho(\mathscr C^{DI\times D[1]})\arrow[u, "\simeq"] \arrow[r, "\ev_{d_0}"'] & \Ho(\mathscr C^{DI})\arrow[u, "\id"']
\end{tikzcd}
\end{equation}
filled with natural isomorphisms. The vertical maps are equivalences, the only non-trivial case being accounted for by Proposition~\ref{prop:pr-equiv}. Accordingly, we can pass to canonical mates (viewing them as right adjoints and choosing trivial units and counits for the outer ones) to get
\begin{equation*}
\begin{tikzcd}[column sep=large]
\Ho(\mathscr C^{DI})\arrow[d, "\id"'] \twocell[dr] & \arrow[l, "\ev_0"'] \Ho(\mathscr C^{D(I\times[1])}) \arrow[d, "\simeq"] \arrow[r, "\ev_1"] & \Ho(\mathscr C^{DI}) \arrow[d,"\id"] \twocell[ld]\\
\Ho(\mathscr C^{DI}) & \arrow[l, "\ev_{d_1}"] \Ho(\mathscr C^{DI\times D[1]}) \arrow[r, "\ev_{d_0}"'] & \Ho(\mathscr C^{DI}).
\end{tikzcd}
\end{equation*}
and moreover these natural transformations are actually isomorphisms. We invert the right hand transformation yielding
\begin{equation}\label{diag:two-cells-mates}
\begin{tikzcd}[column sep=large]
\Ho(\mathscr C^{DI})\arrow[d, "\id"'] \twocell[dr] & \arrow[l, "\ev_0"'] \Ho(\mathscr C^{D(I\times[1])}) \arrow[d, "\simeq"] \arrow[r, "\ev_1"] & \Ho(\mathscr C^{DI}) \arrow[d,"\id"]\\
\Ho(\mathscr C^{DI}) & \arrow[l, "\ev_{d_1}"] \Ho(\mathscr C^{DI\times D[1]}) \arrow[r, "\ev_{d_0}"']\twocell[ur] & \Ho(\mathscr C^{DI}).
\end{tikzcd}
\end{equation}
On the other hand we have a diagram
\begin{equation}\label{diag:two-cells-basis}
\begin{tikzcd}[column sep=large]
\Ho(\mathscr{C}^{DI})\arrow[dd,"\id"'] & \arrow[l, "\ev_{d_1}"'] \Ho(\mathscr{C}^{DI\times D[1]})\arrow[d, "(DI\times i)^*"'] \arrow[r, "\ev_{d_0}"] & \Ho(\mathscr{C}^{DI})\arrow[dd, "\id"]\\
& \arrow[dl, "\ev_0"'] \Ho(\mathscr{C}^{DI\times[1]})\twocell[ur, yshift=-1ex] \arrow[dr, "\ev_1"]\\
\Ho(\mathscr{C}^{DI})\arrow[dr, "\id"']\twocell[rr] & & \Ho(\mathscr{C}^{DI})\arrow[dl, "\id"]\\
& \Ho(\mathscr{C}^{DI})
\end{tikzcd}
\end{equation}
with $i$ from Corollary~\ref{cor:incl-d} and where the transformation in the top right square is the inverse of the natural isomorphism given on $X\colon DI\times D[1]\to\mathscr C$ by $1_*\colon X(\blank, 1:[0]\to[1])\to X(\blank,\id\colon[1]\to[1])$, and the natural transformation populating the lower diamond is the obvious one.

We now define $\tau_D^*$ as the pasting of the diagram
\begin{equation}\label{diag:hod-whole}
\begin{tikzcd}[column sep=large]
&\arrow[dl, "(Df)^*"', bend right=20pt] \Ho(\mathscr{C}^{DJ})\arrow[d, "(Dt)^*"] \arrow[dr, "(Dg)^*", bend left=20pt]\\
\Ho(\mathscr{C}^{DI})\arrow[d, "\id"']\twocell[dr] & \arrow[l, "\ev_0"'] \Ho(\mathscr{C}^{D(I\times[1])} \arrow[d, "\simeq"] \arrow[r, "\ev_1"] &\Ho(\mathscr{C}^{DI})\arrow[d, "\id"]\\
\Ho(\mathscr{C}^{DI})\arrow[dd, "\id"'] & \arrow[l, "\ev_{d_1}"] \Ho(\mathscr{C}^{DI\times D[1]})\twocell[ur] \arrow[d,"(DI\times i)^*"'] \arrow[r, "\ev_{d_0}"'] & \Ho(\mathscr{C}^{DI})\arrow[dd, "\id"]\\
& \arrow[dl, "\ev_0"']\Ho(\mathscr{C}^{DI\times[1]})\twocell[ur, yshift=-1ex] \arrow[dr, "\ev_1"]\\
\Ho(\mathscr{C}^{DI})\twocell[rr]\arrow[dr, "\id"'] & & \Ho(\mathscr{C}^{DI})\arrow[dl, "\id"]\\
& \Ho(\mathscr{C}^{DI})
\end{tikzcd}
\end{equation}
obtained by stacking (\ref{diag:two-cells-top}), (\ref{diag:two-cells-mates}), and (\ref{diag:two-cells-basis}) atop each other.
\end{constr}

\begin{prop}\label{prop:2-cells-natural}
The above definition extends the $1$-functor $\HOD({\mathscr C})$ to a prederivator. Moreover, the maps $p^*$ assemble into a strict morphism of prederivators $\HOcof({\mathscr C})\to\HOD({\mathscr C})$, and this morphism is an equivalence.
\begin{proof}
We observe that $p^*$ is a (strict) natural transformation of the underlying $1$-functors by naturality, and a levelwise equivalence by Proposition~\ref{prop:cof-hoder-models}-(\ref{item:p-star}). We will now show that for each natural transformation $\tau\colon f\Rightarrow g$ of functors $f,g\colon I\to J$ between small categories the two pastings
\begin{equation}\label{diag:pastings-d1}
\begin{tikzcd}
\Ho(\mathscr{C}^J)\arrow[r, "p^*"] & \Ho(\mathscr{C}^{DJ}) \arrow[r, "(Df)^*"{name=t}, bend left=20pt] \arrow[r, bend right=20pt, "(Dg)^*"'{name=b}] \twocell[from=t, to=b] &[1em] \Ho(\mathscr{C}^{DI})
\end{tikzcd}
\end{equation}
and
\begin{equation}\label{diag:pastings-d2}
\begin{tikzcd}
\Ho(\mathscr{C}^J) \arrow[r, "f^*"{name=t}, bend left=20pt] \arrow[r, bend right=20pt, "g^*"'{name=b}] \twocell[from=t, to=b] &[1em] \Ho(\mathscr{C}^{I})\arrow[r, "p^*"] & \Ho(\mathscr{C}^{DI})
\end{tikzcd}
\end{equation}
agree. From this all the remaining claims follow: namely, a natural transformation between prescribed functors is determined by what it does on an essentially wide subcategory. Thus we can conclude from the $2$-functoriality of $\HOcof({\mathscr C})$ that the above turns $\HOD({\mathscr C})$ into a strict $2$-functor. With this established, the equality of the pastings (\ref{diag:pastings-d1}) and (\ref{diag:pastings-d2}), together with the opening remark proves $2$-naturality, i.e.~the $p^*$ form a strict morphism of derivators, and this was already seen to be a levelwise equivalence.

To prove that the pastings indeed agree, we note that by naturality of $p$ we have a commutative diagram
\begin{equation}\label{diag:hodcof-top}
\begin{tikzcd}[column sep=.4em, row sep=1.5em]
&[-2.5pt] &[-1pt] & \arrow[lldd, "(Df)^*"' near start, bend right=25pt] \Ho(\mathscr{C}^{DJ})\arrow[dd] \arrow[rrdd, "(Dg)^*" near start, bend left=25pt] \\
& & \arrow[lldd, "f^*"', bend right=30pt, crossing over] \Ho(\mathscr{C}^J)\arrow[ru] \arrow[rrdd, "g^*", bend left=30pt, crossing over]\\
& \Ho(\mathscr{C}^{DI}) & &[-15pt] \arrow[ll]\Ho(\mathscr{C}^{D(I\times[1])})\arrow[shorten <= 1em, rr] & &[-2.5pt] \Ho(\mathscr{C}^{DI}).\\
\Ho(\mathscr{C}^I)\arrow[ur] & & \arrow[ll, "\ev_0"] \Ho(\mathscr{C}^{I\times[1]})\arrow[from=uu, crossing over]\arrow[ur]\arrow[rr, "\ev_1"'] & &[-2pt] \Ho(\mathscr{C}^I)\arrow[ur]
\end{tikzcd}
\end{equation}
Moreover, we have a coherent diagram
\begin{equation}\label{diag:cube-before-mate}
\begin{tikzcd}[column sep=.25em, row sep=1.5em]
& \Ho(\mathscr{C}^{DI})\arrow[from=dd] & & \arrow[ll, "\ev_0"']\Ho(\mathscr{C}^{D(I\times[1])})\arrow[from=dd] \arrow[rr, "\ev_1"] & & \Ho(\mathscr{C}^{DI})\arrow[from=dd]\\
\Ho(\mathscr{C}^I)\arrow[ur] & & \arrow[ll, crossing over]\Ho(\mathscr{C}^{I\times[1]})\arrow[ur]\arrow[rr, crossing over] & & \Ho(\mathscr{C}^I)\arrow[ur]\\
& \Ho(\mathscr{C}^{DI}) & & \arrow[ll, "\ev_{d_1}" near start]\Ho(\mathscr{C}^{DI\times D[1]}) \arrow[rr, "\ev_{d_0}"' near start] & & \Ho(\mathscr{C}^{DI})\\ 
\Ho(\mathscr{C}^I)\arrow[uu]\arrow[ur] & & \arrow[ll, "\ev_0"]\Ho(\mathscr{C}^{I\times[1]})\arrow[uu, crossing over]\arrow[rr, "\ev_1"']\arrow[ur] & & \Ho(\mathscr{C}^I)\arrow[ur]\arrow[uu, crossing over]
\end{tikzcd}
\end{equation}
where the front face is the obvious one, the back face is (\ref{diag:two-cells-before}), and the front-to-back maps are either $p^*$ or $(p\times p)^*$. Note that we omitted the non-trivial $2$-cells on the back face for readability. Indeed, by Construction~\ref{constr:homotopic-eval} it suffices to show that for each $j\colon[0]\to[1]$ the diagram
\begin{equation*}
\begin{tikzcd}[column sep=.25em, row sep=1.5em]
& \Ho(\mathscr{C}^{D(I\times[1])})\arrow[from=dd] \arrow[rr, "D(\id\times j)^*"] && \Ho(\mathscr{C}^{D(I\times[0])})\arrow[from=dd]\arrow[rr]&&\Ho(\mathscr{C}^{DI})\\
\Ho(\mathscr{C}^{I\times[1]}) \arrow[ur]\arrow[rr, crossing over, "\ev_{j(0)}" near end] && \Ho(\mathscr{C}^{I\times[0]})\arrow[ur]\arrow[rr, crossing over,  end anchor={[xshift=5pt]}]&&\hbox to 5pt{\hfill} \Ho(\mathscr{C}^I)\arrow[ur]\\
& \Ho(\mathscr{C}^{DI\times D[1]})\arrow[rr]&&[-10pt]\Ho(\mathscr{C}^{DI\times D[0]})\arrow[rr, "\ev_{\id_{[0]}}"' near start]&& \Ho(\mathscr{C}^{DI})\arrow[uu]\\
\Ho(\mathscr{C}^{I\times[1]})\arrow[uu]\arrow[ur]\arrow[rr, "\ev_{j(0)}"'] &&[-13pt] \Ho(\mathscr{C}^{I\times[0]})\arrow[ur]\arrow[uu, crossing over]\arrow[rr, end anchor={[xshift=5pt]}] &&\hbox to 5pt{\hfill}\Ho(\mathscr{C}^I)\arrow[uu, crossing over, xshift=2.5pt]\arrow[ur]
\end{tikzcd}
\end{equation*}
is coherent. Here the back faces are given by (\ref{diag:h-e-easy}) respectively (\ref{diag:h-e-mate}) and the front-to-back maps are defined as before. In particular, all the faces of the (commutative) cube on the left are filled with the identity transformations, and so it suffices to prove coherence of the right hand cube. This amounts to saying that the natural transformation from (\ref{diag:h-e-mate}) is the identity on diagrams in $\im (p\times p)^*$. For this we consider the commutative diagram
\begin{equation*}
\begin{tikzcd}[column sep=4.5em]
\Ho(\mathscr{C}^{I\times[0]})\arrow[d, "(p\times p)^*"'] & \arrow[l, "\pr_1^*"'] \Ho(\mathscr{C}^I)\arrow[d,"p^*"]\\
\Ho(\mathscr{C}^{DI\times D[0]}) \arrow[d, "{(D\pr_1,D\pr_2)^*}"'] & \arrow[l, "\pr_1^*"'] \Ho(\mathscr{C}^{DI})\arrow[d, "\id"]\\
\Ho(\mathscr{C}^{D(I\times [0])}) & \arrow[l, "(D\pr_1)^*"] \Ho(\mathscr{C}^{DI}).
\end{tikzcd}
\end{equation*}
The top map is an isomorphism and hence we can make it into the right adjoint in an adjoint equivalence where both unit and counit are the identity. It is now trivial to check that with respect to this and the  adjunctions fixed in Construction~\ref{constr:homotopic-eval} the canonical mate of the top square and the canonical mate of the total rectangle are both the identity (because all the relevant counits and units are); on the other hand (\ref{diag:h-e-mate}) is by construction the canonical mate of the bottom rectangle, and hence the compatibility of mates with pasting implies the claim.

Appealing to the compatibility of mates with pasting again we get from (\ref{diag:cube-before-mate}), after inverting a natural isomorphism, the coherent diagram
\begin{equation}\label{diag:cube-after-mate}
\begin{tikzcd}[column sep=.25em, row sep=1.5em]
& \Ho(\mathscr{C}^{DI})\arrow[dd] & & \arrow[ll, "\ev_0"']\Ho(\mathscr{C}^{D(I\times[1])})\arrow[dd] \arrow[rr, "\ev_1"] & & \Ho(\mathscr{C}^{DI})\arrow[dd]\\
\Ho(\mathscr{C}^I)\arrow[ur] & & \arrow[ll, crossing over]\Ho(\mathscr{C}^{I\times[1]})\arrow[ur]\arrow[rr, crossing over] & & \Ho(\mathscr{C}^I)\arrow[ur]\\
& \Ho(\mathscr{C}^{DI}) & & \arrow[ll, "\ev_{d_1}" near start]\Ho(\mathscr{C}^{DI\times D[1]}) \arrow[rr, "\ev_{d_0}"' near start] & & \Ho(\mathscr{C}^{DI})\\ 
\Ho(\mathscr{C}^I)\arrow[from=uu, crossing over]\arrow[ur] & & \arrow[ll, "\ev_0"]\Ho(\mathscr{C}^{I\times[1]})\arrow[from=uu, crossing over]\arrow[rr, "\ev_1"']\arrow[ur] & & \Ho(\mathscr{C}^I)\arrow[ur]\arrow[from=uu, crossing over]
\end{tikzcd}
\end{equation}
where the back face is now (\ref{diag:two-cells-mates}) and the middle face might be filled with a nontrivial isomorphism.

Moreover, the diagram
\begin{equation}\label{diag:cube-i}
\begin{tikzcd}[column sep=.25em, row sep=1.5em]
& \Ho(\mathscr{C}^{DI})\arrow[dd] & & \arrow[ll, "\ev_{d_1}"']\Ho(\mathscr{C}^{DI\times D[1]})\arrow[dd, "(DI\times i)^*" near start] \arrow[rr, "\ev_{d_0}"] & & \Ho(\mathscr{C}^{DI})\arrow[dd]\\
\Ho(\mathscr{C}^I)\arrow[ur] & & \arrow[ll, crossing over]\Ho(\mathscr{C}^{I\times[1]})\arrow[ur]\arrow[rr, crossing over] & & \Ho(\mathscr{C}^I)\arrow[ur]\\
& \Ho(\mathscr{C}^{DI}) & & \arrow[ll, "\ev_0" near start]\Ho(\mathscr{C}^{DI\times [1]}) \arrow[rr, "\ev_1"' near start] & & \Ho(\mathscr{C}^{DI})\\ 
\Ho(\mathscr{C}^I)\arrow[from=uu, crossing over]\arrow[ur] & & \arrow[ll, "\ev_0"]\Ho(\mathscr{C}^{I\times[1]})\arrow[from=uu, crossing over]\arrow[rr, "\ev_1"']\arrow[ur] & & \Ho(\mathscr{C}^I)\arrow[ur]\arrow[from=uu, crossing over]
\end{tikzcd}
\end{equation}
with the back face from (\ref{diag:two-cells-basis}) is also coherent: this is trivial for the left hand cube (which is actually commutative), and for the hand right cube it follows from the explicit description of the natural transformation and the formula $p(1:[0]\to[1])=\id$.

Finally, by $2$-functoriality of $\HOcof({\mathscr C})$ the two pastings
\begin{equation*}
\begin{tikzcd}
\Ho(\mathscr{C}^{I\times[1]})\arrow[r, "(p\times\id)^*"] &[.5em] \Ho(\mathscr{C}^{DI\times[1]}) \arrow[r, "\ev_0"{name=t,xshift=-2pt}, bend left=20pt] \arrow[r, bend right=20pt, "\ev_1"'{name=b}] \twocell[from=t, to=b] & \Ho(\mathscr{C}^{DI})
\end{tikzcd}
\end{equation*}
and
\begin{equation*}
\begin{tikzcd}
\Ho(\mathscr{C}^{I\times[1]}) \arrow[r, "\ev_0"{name=t,above,xshift=-2pt}, bend left=20pt] \arrow[r, bend right=20pt, "\ev_1"{name=b,below}] \twocell[from=t, to=b] &[.5em] \Ho(\mathscr{C}^{I}) \arrow[r, "p^*"] & \Ho(\mathscr{C}^{DI})
\end{tikzcd}
\end{equation*}
agree, i.e.~we have a coherent diagram
\begin{equation}\label{diag:hodcof-bot}
\begin{tikzcd}[column sep=.45em, row sep=1em]
& \Ho(\mathscr{C}^{DI})\arrow[ddrr, bend right=30pt] && \arrow[ll, "\ev_0"'] \Ho(\mathscr{C}^{DI\times[1]}) \arrow[rr, "\ev_1"] && \Ho(\mathscr{C}^{DI})\arrow[ddll, bend left=30pt]\\
\Ho(\mathscr{C}^I)\arrow[ur]\arrow[ddrr, bend right=30pt] && \arrow[ll, "\ev_0", crossing over] \Ho(\mathscr{C}^{I\times[1]})\arrow[ur] \arrow[rr, "\ev_1"'] && \Ho(\mathscr{C}^I)\arrow[ur]\arrow[ddll, bend left=30pt, crossing over]\\
& & & \Ho(\mathscr{C}^{DI})\\
& & \Ho(\mathscr{C}^I)\arrow[ur]
\end{tikzcd}
\end{equation}
where front and back face are filled with the obvious natural transformations going from left to right.

Now we stack the diagrams (\ref{diag:hodcof-top}), (\ref{diag:cube-after-mate}), (\ref{diag:cube-i}), and (\ref{diag:hodcof-bot}) together from top to bottom. This yields a coherent diagram whose front face represents $\tau^*$. Moreover its back face is given by stacking together the diagrams (\ref{diag:two-cells-top}), (\ref{diag:two-cells-mates}), and (\ref{diag:two-cells-basis}), so it agrees with (\ref{diag:hod-whole}) and hence its pasting is by definition $\tau_D^*$. 

The front to back map at the top is given by $p^*\colon \Ho(\mathscr{C}^J)\to\Ho(\mathscr{C}^{DJ})$ and the one at the bottom by $p^*\colon\Ho(\mathscr{C}^I)\to\Ho(\mathscr{C}^{DI})$. Moreover, the remaining outer faces are all filled with identity transformations by construction, so this large diagram precisely witnesses the equality of (\ref{diag:pastings-d1}) and (\ref{diag:pastings-d2}), finishing the proof.
\end{proof}
\end{prop}

\begin{rk}
We remark that a purely formal calculation shows that there is a unique way to extend the strict $1$-functor $\HOD(\mathscr C)$ to a strict $2$-functor in such a way that $p^*$ gives rise to a $2$-natural equivalence $\HOcof(\mathscr C)\to\HOD(\mathscr C)$; similar remarks apply to the equivalences from Proposition~\ref{prop:hond-vs-hod} respectively Proposition~\ref{prop:final-equivalence}.

However, using this abstract statement instead of the above comes at the cost of the explicit description of the action of $\HOD(\mathscr C)$ on $2$-cells. Since we ultimately want to establish a zig-zag of equivalences between $\HOcof(\mathscr C)$ and $\HOinf(\nerve_f(\mathscr C))$, which both already come equipped with $2$-functor structures, we will at some point have to check $2$-naturality of at least one map directly. We think that the above indirectness would make this proof much more complicated and lengthy than any of the individual verifications. Accordingly, we decided to give rather explicit constructions of the desired $2$-cells, which might also be interesting in their own right.
\end{rk}

Obviously $\HOD$ becomes a strict $1$-functor $\cat{COFCAT}\to\cat{PREDER}$ via pushforward. We remark that the above constitutes a strictly natural transformation $\HOcof\Rightarrow\HOD$.

\subsection{Homotopy prederivators of associated quasi-categories}
We will now introduce yet another auxiliary prederivator that ``interpolates'' between $\HOD(\mathscr C)$ and $\HOinf(\nerve_f(\mathscr C))$. As before we construct it first as a $1$-functor:

\begin{defi}
Let $\mathscr C$ be a cofibration category. We define $\HOND(\mathscr C)$ to be the $1$-functor $\cat{Cat}^\op\to\cat{CAT}$ given by $I\mapsto\h\nerve_f({\mathscr C}^{DI}_{\text{R}})$ together with the obvious restriction maps.
\end{defi}

To compare this to $\HOD$ we will use:

\begin{constr}\label{constr:theta}
Let $\mathscr C$ be a cofibration category. Following Szumi\l{}o, we construct a map $\theta\colon\h\nerve_f({\mathscr C})\to\Ho({\mathscr C})$ as follows: an object $X\colon D[0]\to{\mathscr C}$ is sent to $X(\id_{[0]})$ and a morphism represented by an edge $F\colon D[1]\to\mathscr C$ is sent to the zig-zag 
\begin{equation*}
\begin{tikzcd}
F(0\colon [0]\to[1]) \arrow[r, "0_*"] & F(\id\colon[1]\to[1]) & \arrow[l, "1_*"', "\sim"] F(1\colon[0]\to[1]).
\end{tikzcd}
\end{equation*}
\end{constr}

\begin{lemma}\label{lemma:theta-natural}
The above map is well-defined and an equivalence. It commutes strictly with push-forward along exact functors.
\begin{proof}
The naturality claim is trivial and the rest of the statement is \cite[proof of Lemma 4.10]{szumilo}.
\end{proof}
\end{lemma}

We now want to extend $\HOND$ over $2$-cells. For this we will need:

\begin{constr}\label{constr:e}
We define a functor $e\colon \h\nerve_f({\mathscr C}_{\text{R}}^{DI\times D[0]})\to\h\nerve_f({\mathscr C}_{\text{R}}^{DI})$ as follows: an object corresponding to $X\colon D[0]\times DI\times D[0]\to\mathscr C$ (where the first factor comes from the definition of $\nerve_f$) is sent to the object corresponding to $\widetilde X\colon D[0]\times DI\to\mathscr C$ with $\widetilde X([m]\to [0], f\colon [n]\to I) = X([m]\to [0], f, [m]\to [0])$, and the class of a $1$-simplex corresponding to $F\colon D[1]\times DI\times D[0]\to\mathscr C$ is sent to the class of the edge corresponding to $\widetilde F\colon D[1]\times DI\to\mathscr C$ with $\widetilde F(f\colon [m]\to[1], g\colon [n]\to I)$ = $F(f, g, [m]\to[0])$. 

We emphasize that $e$ is \emph{not} induced from any (exact) functor $\mathscr{C}^{DI\times D[0]}_{\text{R}}\to\mathscr{C}_{\text{R}}^{DI}$, but rather it ``mixes in'' the $D[0]$- resp.~$D[1]$-factor coming from the definition of $\nerve_f$.
\end{constr}

\begin{lemma}\label{lemma:commutative-face}
The above map $e$ is a well-defined functor. Moreover, the diagram
\begin{equation*}
\begin{tikzcd}[column sep=large]
\h\nerve_f(\mathscr{C}_{\textup{R}}^{DI\times D[0]}) \arrow[d, "\incl\circ\theta"', "\simeq"] \arrow[r, "e"] & \h\nerve_f(\mathscr{C}_{\textup{R}}^{DI})\arrow[d, "\incl\circ\theta", "\simeq"']\\
\Ho(\mathscr{C}^{DI\times D[0]}) \arrow[r, "\ev_{[0]\to[0]}"']& \Ho(\mathscr{C}^{DI})
\end{tikzcd}
\end{equation*}
commutes.
\begin{proof}
We first observe that the diagrams $\widetilde X$ and $\widetilde F$ are indeed Reedy cofibrant and homotopical: we show this for the second one, the proof of the first one is analogous. For this we note that $\widetilde F$ is the image of $F$ under the restriction along $D[1]\times DI\cong D\big([1]\times[0]\big)\times DI\to D[1]\times D[0]\times DI\cong D[1]\times DI\times D[0]$ and restriction along each of these maps preserves Reedy cofibrancy and homotopicalness (cf.~Proposition~\ref{prop:pr-equiv}).

Next, we will show that for each homotopical $F\colon D[1]\times DI\times D[0]\to\mathscr C$ the zig-zags
\begin{equation}\label{eq:upper-zig-zag}
\begin{aligned}
F(0\colon[0]\to[1],\blank,[0]\to[0])&\to
F(\id\colon[1]\to[1],\blank,[1]\to[0])\cr&\gets
F(1\colon[0]\to[1],\blank,[0]\to[0])
\end{aligned}
\end{equation}
and
\begin{equation}\label{eq:lower-zig-zag}
\begin{aligned}
F(0\colon[0]\to[1],\blank,[0]\to[0])&\to
F(\id\colon[1]\to[1],\blank,[0]\to[0])\cr&\gets
F(1\colon[0]\to[1],\blank,[0]\to[0])
\end{aligned}
\end{equation}
define the same morphism in $\Ho({\mathscr C}^{DI})$. Indeed, both give rise to natural transformations between the same two functors $\Ho({\mathscr C}^{D[1]\times DI\times D[0]})\to\Ho({\mathscr C}^{DI})$ and since $(p\times p\times p)^*\colon\Ho({\mathscr C}^{[1]\times I\times [0]})\to\Ho({\mathscr C}^{D[1]\times DI\times D[0]})$ is an equivalence by the exponential law, it suffices to prove this for diagrams in the image of $(p\times p\times p)^*$, which is trivial.

From this all of the claims follow: to see that the class $e(F)$ is independent of the choice of representative, it suffices to prove this for the image in $\Ho({\mathscr C}^{DI})$, because the right hand vertical map is an equivalence. But this is precisely the zig-zag (\ref{eq:upper-zig-zag}), which by the above agrees with (\ref{eq:lower-zig-zag}), which in turn is precisely the image of $[F]$ under the lower left composition, proving the claim. The same argument shows functoriality and with this established the above precisely shows commutativity.
\end{proof}
\end{lemma}

\begin{constr}\label{constr:hond-trafo}
We consider the diagram
\begin{equation}\label{diag:transfer-iso}
\begin{tikzcd}[column sep=0pt]
& \Ho(\mathscr{C}^{DI\times D[0]}) \arrow[rr, "\ev_{[0]\to[0]}"]\arrow[dd] & & \Ho(\mathscr{C}^{DI}) \arrow[dd]\\
\h\nerve_f(\mathscr{C}_{\textup{R}}^{DI\times D[0]}) \arrow[ur]\arrow[rr, "e" near end, crossing over] & & \h\nerve_f(\mathscr{C}_{\textup{R}}^{DI}) \arrow[ur]\\
& \Ho(\mathscr{C}^{D(I\times[0])})\arrow[rr] & & \Ho(\mathscr{C}^{DI})\\
\h\nerve_f(\mathscr{C}_{\textup{R}}^{D(I\times[0])})\arrow[from=uu]\arrow[rr]\arrow[ru] & & \h\nerve_f(\mathscr{C}_{\textup{R}}^{DI})\arrow[from=uu, crossing over]\arrow[ur]
\end{tikzcd}
\end{equation}
where the front-to-back maps are induced from $\theta$ and the obvious inclusion. All of the faces containing them commute strictly (the only non-trivial case is accounted for by Lemma~\ref{lemma:commutative-face}) and moreover all of the front-to-back maps are equivalences. Hence there is a unique natural transformation filling the front face such that if we fill the back face with the natural isomorphism from Construction~\ref{constr:homotopic-eval} the resulting diagram is coherent, and this transformation is an isomorphism.

For a natural transformation $\tau\colon f\Rightarrow g$ of functors $f,g\colon I\to J$ between small categories (with corresponding functor $t\colon I\times[1]\to J$) we now define $\tau^D_{\text{N}}$ as the pasting
\begin{equation*}
\begin{tikzcd}[column sep=small, cramped]
& & \arrow[dll, "(Df)^*"', bend right=10pt]\h\nerve_f(\mathscr{C}_{\textup{R}}^{DJ})\arrow[d, "(Dt)^*"] \arrow[drr, "(Dg)^*", bend left = 10pt]\\[5pt]
\h\nerve_f(\mathscr{C}_{\textup{R}}^{DI})\twocell[dr] \arrow[d,"\id"']& \arrow[l] \h\nerve_f(\mathscr{C}_{\textup{R}}^{D(I\times[0])}) \twocell[dr] \arrow[d, "\simeq"'] & \arrow[l] \h\nerve_f(\mathscr{C}_{\textup{R}}^{D(I\times[1])}) \arrow[d, "\simeq"] \arrow[r] & \h\nerve_f(\mathscr{C}_{\textup{R}}^{D(I\times[0])})\twocell[from=dl] \arrow[d, "\simeq"] \arrow[r] & \h\nerve_f(\mathscr{C}_{\textup{R}}^{DI})\twocell[from=dl]\arrow[d, "\id"]\\
\h\nerve_f(\mathscr{C}_{\textup{R}}^{DI})\arrow[dd, "\id"'] & \arrow[l,"e"] \h\nerve_f(\mathscr{C}_{\textup{R}}^{DI\times D[0]}) & \arrow[l] \h\nerve_f(\mathscr{C}_{\textup{R}}^{DI\times D[1]})\arrow[d, "\Phi"] \arrow[r] & \h\nerve_f(\mathscr{C}_{\textup{R}}^{DI\times D[0]}) \arrow[r, "e"'] &\h\nerve_f(\mathscr{C}_{\textup{R}}^{DI})\arrow[dd, "\id"]\\
& & \arrow[lld, "\ev_0"']\h\big(\nerve_f(\mathscr{C}_{\textup{R}}^{DI})^{\Delta^1}\big)\arrow[rrd, "\ev_1"]\\
\h\nerve_f(\mathscr{C}_{\textup{R}}^{DI}) \twocell[rrrr]\arrow[rrd, "\id"']& & & & \h\nerve_f(\mathscr{C}_{\textup{R}}^{DI})\arrow[lld, "\id"]\\
& & \h\nerve_f(\mathscr{C}_{\textup{R}}^{DI})
\end{tikzcd}
\end{equation*}
where the inner natural transformations in the second row from the top are the canonical mates of the respective identity transformations, viewing the vertical maps as right adjoints. Moreover, the outer ones are obtained in the same way from the isomorphism in (\ref{diag:transfer-iso}) respectively its inverse; here we have again chosen trivial units and counits for the outermost arrows. The map $\Phi$ comes from Construction~\ref{constr:frame-diag}. Finally, the natural transformation in the lower diamond is the obvious one. We remark that the trapezoids involving $e$ as one of the top edges do indeed commute---this is the reason for the rather strange definition of $e$.
\end{constr}

We can now compare this to the previous intermediate prederivator:

\begin{prop}\label{prop:hond-vs-hod}
The above construction makes $\HOND(\mathscr C)$ into a prederivator. Moreover, the maps $\incl\circ\theta$ assemble into an equivalence of prederivators $\HOND(\mathscr C)\to\HOD({\mathscr C})$.
\begin{proof}
By Lemma~\ref{lemma:theta-natural} and Proposition~\ref{prop:cof-hoder-models}-(\ref{item:reedy-repl}) this is a levelwise equivalence and strictly compatible with restrictions. Hence by arguments analogous to the proof of Proposition~\ref{prop:2-cells-natural} it suffices to check compatibility with $2$-cells. 

By construction and the same arguments as before it is enough to prove that the pastings
\begin{equation*}
\begin{tikzcd}
\h\nerve_f(\mathscr{C}_{\textup{R}}^{DI\times D[1]})\arrow[d]\\
\Ho(\mathscr{C}_{\textup{R}}^{DI\times D[1]})\arrow[d, "(DI\times i)^*"'] \arrow[r] & \Ho(\mathscr{C}^{DI}),\arrow[ddl, bend left=30pt]\\
\Ho(\mathscr{C}^{DI\times[1]})\arrow[d, "\ev_0"'{name=s}, bend right=40pt]\arrow[d, "\ev_1"{name=t}, bend left=40pt]\twocell[from=s,to=t]\twocell[ur,yshift=-4ex]\\[3ex]
\Ho(\mathscr{C}^{DI})
\end{tikzcd}
\end{equation*}
where the lower portion comes from (\ref{diag:hod-whole}), and
\begin{equation*}
\begin{tikzcd}
\h\nerve_f(\mathscr{C}_{\textup{R}}^{DI\times D[1]})\arrow[r] & \h\big(\nerve_f(\mathscr{C}_{\textup{R}}^{DI})^{\Delta^1}\big)\arrow[r, "\ev_0"{above, name=s}, bend left=20pt]\arrow[r, "\ev_1"{below, name=t, xshift=1pt}, bend right=20pt] \twocell[from=s, to=t] & \h\nerve_f(\mathscr{C}_{\textup{R}}^{DI})\arrow[r] & \Ho(\mathscr{C}^{DI})
\end{tikzcd}
\end{equation*}
agree. Now an explicit computation shows that on an object corresponding to $F\colon D[0]\times DI\times D[1]\to\mathscr C$ the first one is given by the zig-zag
\begin{equation*}
\begin{aligned}
F([0]\to[0],\blank,0\colon [0]\to[1])&\to F([0]\to[0],\blank, \id\colon[1]\to[1])\\
&\gets F([0]\to[0],\blank,1\colon[0]\to[1])
\end{aligned}
\end{equation*}
whereas the second one is the zig-zag
\begin{equation*}
\begin{aligned}
F([0]\to[0],\blank,0\colon [0]\to[1])&\to F([1]\to[0],\blank, \id\colon[1]\to[1])\\
&\gets F([0]\to[0],\blank,1\colon[0]\to[1])
\end{aligned}
\end{equation*}
and the claim follows from the proof of Lemma~\ref{lemma:commutative-face}.
\end{proof}
\end{prop}

Again we remark that $\HOND$ becomes a strict functor $\cat{COFCAT}\to\cat{PREDER}$ via pushforward and the above map constitues a strictly natural transformation $\HOND\Rightarrow\HOD$.

\subsection{Sidestepping Reedy cofibrancy}
The maps $\Phi$ from Construction~\ref{constr:frame-diag} provide a natural candidate for an equivalence $\HOND(\mathscr C)\to \HOinf(\nerve_f(\mathscr C))$ and it is easy to see that this is a strictly natural transformation of underlying $1$-functors. However, proving $2$-naturality becomes surprisingly hard because of the indirectness in the above definition of the action of $\HOND(\mathscr C)$ on natural transformations.

To solve this issue it will be useful to introduce a variant of $\nerve_f$ (or rather its homotopy category) that is based on diagrams that are merely homotopical without any Reedy cofibrancy assumptions.

\begin{prop}
Let $\mathscr C$ be a cofibration category. The functor $\cat{SSET}\to\cat{SET}$
\begin{equation*}
K\mapsto \{\hbox{homotopical diagrams $DK\to\mathscr C$}\}
\end{equation*}
(together with the obvious restriction maps) is representable by a large simplicial set $\widetilde{\nerve_h}({\mathscr C})$ given explicitely by
\begin{equation*}
\widetilde{\nerve_h}({\mathscr C})_n = \{\hbox{homotopical diagrams $D[n]\to\mathscr C$}\}.
\end{equation*}
\begin{proof}
This is immediate from Lemma~\ref{lemma:representability}.
\end{proof}
\end{prop}

Note that we have a natural inclusion $\nerve_f({\mathscr C})\hookrightarrow\widetilde{\nerve_h}({\mathscr C})$.

\begin{defi}
We define the category $\h\nerve_h({\mathscr C})$ to be ``the'' strict localization of $\h\widetilde{\nerve_h}({\mathscr C})$ with respect to the homotopical diagrams $D\widehat{[1]}\to\mathscr C$.
\end{defi}

\begin{warn}
We use $\h\nerve_h$ merely as a primitive symbol here; the same remark applies to Definition~\ref{defi:nf-nh-general} below.
\end{warn}

\begin{rk}
By the usual presentation of the homotopy category of a simplicial set and the universal property of localization a functor $F\colon\h\nerve_h(\mathscr{C})\to\mathscr{D}$ corresponds to assigning to each vertex $x\colon D[0]\to\mathscr{C}$ an object $Fx$ and to each edge $f$ from $x$ to $y$ a morphism $Ff\colon Fx\to Fy$ such that the following conditions are satisfied:
\begin{enumerate}
\item $F$ is compatible with $2$-cells in the obvious way.
\item $F$ sends degenerate edges to identity morphisms.
\item $F$ sends any edge of the form $D\widehat{[1]}\to\mathscr{C}$ to an isomorphism.
\end{enumerate}
We remark that the second condition is vacuous in the presence of the other two: namely, let $f$ be the image of a degenerate edge. Then $f^2=f$ by (1) and on the other hand $f$ is an isomorphism by (3). We conclude that $f$ is an identity arrow as desired.
\end{rk}

\begin{lemma}\label{lemma:nf-nh-special}
The composition $i\colon\h\nerve_f({\mathscr C})\hookrightarrow\h\widetilde{\nerve_h}({\mathscr C})\to\h\nerve_h({\mathscr C})$ is an equivalence of categories.
\begin{proof}
We will construct a quasi-inverse $q$. For this let $X\colon D[0]\to\mathscr C$ be homotopical. By cofibrant replacement we can choose a Reedy cofibrant homotopical diagram $q(X)\colon D[0]\to\mathscr C$ together with a weak equivalence $\sigma_X\colon q(X)\to X$; we choose $\sigma_X=\id$ whenever $X$ is already Reedy cofibrant.

Now let $f\colon D[1]\to\mathscr C$ be a morphism from $X$ to $Y$. Since $D(\partial[1])\to D[1]$ is a sieve, Lemma~\ref{lemma:relative-cofibrant-replacement} allows us to find a Reedy cofibrant diagram $q(f)\colon D[1]\to\mathscr C$ together with a weak equivalence $\sigma_f\colon q(f)\to f$ in ${\mathscr C}^{D[1]}$ such that $\sigma_f$ restricts to $\sigma_X$ respectively $\sigma_Y$ on the boundary; in particular $q(f)$ restricts to $q(X)$ respectively $q(Y)$. Again we take $\sigma_f=\id$ whenever $f$ was already Reedy cofibrant.

We now claim that $q$ is well-defined and a functor. Indeed, if $R\colon D[2]\to{\mathscr C}$ is homotopical, we can choose by another application of Lemma~\ref{lemma:relative-cofibrant-replacement} a Reedy cofibrant replacement $\widehat R$ extending the replacement on $D(\partial\Delta^2)$ chosen above. $\widehat R$ then witnesses the desired relation in $\h\nerve_f({\mathscr C})$. It remains to show that $q$ sends weak equivalences to isomorphisms. But, indeed, if $f\colon D[1]\to\mathscr C$ is homotopical with respect to the maximal homotopical structure then so is $q(f)$ by $2$-out-of-$3$ and the claim follows from Lemma~\ref{lemma:eq-in-nf}.

We note that $qi=\id$ by construction and we will now prove $iq\cong\id$. For this we define a natural transformation $\tau\colon iq\Rightarrow\id$ as follows: $\tau_X$ is the composition $D[1]\cong D([1]\times [0])\to D[1]\times D[0]\to [1]\times D[0]\to \mathscr C$ where the last map is adjunct to $\sigma_X$. We note that this is indeed an equivalence because we have just localized at such maps. To see naturality it suffices to prove compatibility with any morphism $f\colon X\to Y$ coming from an actual edge of $\widetilde{\nerve_h}({\mathscr C})$.
For this we consider the composition $D([1]\times[1])\to D[1]\times D[1]\to [1]\times D[1]\to\mathscr C$ where the last map is adjunct to $\sigma_f$. Viewing this as $\Delta^1\times\Delta^1\to\widetilde{\nerve_h{\mathscr C}}$ exhibits the desired commutativity.
\end{proof}
\end{lemma}

\begin{defi}
We denote by $\cat{HCOFCAT}$ the full subcategory of $\cat{CATWE}$ spanned by the cofibration categories.
\end{defi}

\begin{cor}\label{cor:nh-nice}
\begin{enumerate}
\item Pushforward makes $\h\nerve_h$ into a functor $\cat{HCOFCAT}\to\cat{CAT}$.
\item The above map provides a strictly natural equivalence between the restriction of this functor to $\cat{COFCAT}$ and the composition $\h\circ\nerve_f$.
\item The functor $\h\nerve_h$ sends weak equivalences of cofibration categories to equivalences.\label{item:nh-homotopical}
\end{enumerate}
\begin{proof}
The first statement is obvious and so is naturality in the second statement. The remaining part of (2) is Lemma~\ref{lemma:nf-nh-special} and the third statement now follows by $2$-out-of-$3$.
\end{proof}
\end{cor}

We will also need a more general variant of the above:

\begin{defi}\label{defi:nf-nh-general}
Let $I$ be a category. We define $\h\big(\nerve_h({\mathscr C})^{\nerve I}\big)$ to be the localization of $\h\big(\widetilde{\nerve_h}({\mathscr C})^{\nerve I}\big)$ with respect to morphisms corresponding to homotopical diagrams $D(\widehat{[1]}\times I)\to\mathscr C$.
\end{defi}

Precisely the same arguments as above show:

\begin{lemma}\label{lemma:nf-nh-general}
The composition $\h\big(\nerve_f({\mathscr C})^{\nerve I}\big)\to\h\big(\widetilde{\nerve_h}({\mathscr C})^{\nerve I}\big)\to\h\big(\nerve_h({\mathscr C})^{\nerve I}\big)$ is an equivalence.\qed
\end{lemma}

We can now use the above to transfer several results about $\h\nerve_f$ to $\h\nerve_h$:

\begin{lemma}
The map $\h\nerve_h({\mathscr C})\to\Ho({\mathscr C})$ defined analogously to Construction~\ref{constr:theta} is well-defined and an equivalence.
\begin{proof}
It suffices to prove that it is well-defined; the remaining part then follows from Lemma~\ref{lemma:theta-natural} together with Lemma~\ref{lemma:nf-nh-special} and $2$-out-of-$3$.

It obviously sends weak equivalences to isomorphisms; hence it only remains to show that for a homotopical diagram $H\colon D[2]\to\mathscr C$ the zig-zags
\begin{equation*}
H(0\colon [0]\to[2]) \to H(d_1\colon [1]\to[2])\gets H(2\colon[0]\to[2])
\end{equation*}
and
\begin{equation*}
\begin{aligned}
H(0\colon [0]\to[2]) &\to H(d_2\colon [1]\to [2])\gets H(1\colon[0]\to[2])\\
&\to H(d_0\colon [1]\to[2])\gets H(2\colon[0]\to[2])
\end{aligned}
\end{equation*}
define the same morphism in $\Ho(\mathscr C)$. But again both sides can be viewed as natural transformations $\ev_0\Rightarrow\ev_2$ of functors $\Ho({\mathscr C}^{D[2]})\to\Ho({\mathscr C})$ and hence it suffices to prove this under the assumption that $H=p^*h$ for some $h\in\mathscr C^{[2]}$. This is obvious.
\end{proof}
\end{lemma}

\begin{lemma}
Let $I$ be a small category. Then the map $\h\nerve_h({\mathscr C}^{DI})\to\h\left(\nerve_h({\mathscr C})^{\nerve I}\right)$ defined analogously to Construction~\ref{constr:frame-diag} is well-defined and an equivalence.
\begin{proof}
As before it suffices to prove that this is well-defined. For this we note that it is induced under $\h$ from the map $\phi\colon\widetilde{\nerve_h}({\mathscr C}^{DI})\to\widetilde{\nerve_h}({\mathscr C})^{\nerve I}$ given degreewise by $(D\pr_1,D\pr_2)^*$, and accordingly it suffices to prove that $\phi$ preserves weak equivalences. This is immediate from the definition.
\end{proof}
\end{lemma}

\subsection{The final comparison step}
We begin by using the results of the previous section to get another interpretation of the isomorphism from (\ref{diag:transfer-iso}).

\begin{lemma}
The map $e\colon\h\nerve_h({\mathscr C}^{DI\times D[0]})\to\h\nerve_h({\mathscr C}^{DI})$ defined analogously to Construction~\ref{constr:e} is well-defined and left-inverse to the map induced from $\pr\colon DI\times D[0]\to DI$. In particular, it is an equivalence.
\begin{proof}
For the first statement we can apply literally the same proof as Lemma~\ref{lemma:commutative-face} once we note that $D\widehat{[1]}\to D\widehat{[1]}\times D[0]$ is homotopical for trivial reasons. The second statement is trivial and the third one now follows from Corollary~\ref{cor:nh-nice}-(\ref{item:nh-homotopical}).
\end{proof}
\end{lemma}

\begin{constr}
We have an adjoint equivalence $e\dashv \pr^*$ where the counit is the identity and an adjoint equivalence $(D\incl)^*\dashv(D\pr)^*$ (where $\pr\colon I\times[0]\to I$ and $\incl\colon I\to I\times[0]$ are the obvious maps) with both unit and counit the identity. Moreover we have a commutative diagram
\begin{equation*}
\begin{tikzcd}
\h\nerve_h(\mathscr{C}^{DI\times D[0]})\arrow[d, "{(D\pr_1, D\pr_2)^*}"'] & \arrow[l, "\pr^*"'] \h\nerve_h(\mathscr{C}^{DI})\arrow[d, "\id"]\\
\h\nerve_h(\mathscr{C}^{D(I\times[0])}) & \arrow[l, "(D\pr)^*"]\h\nerve_h(\mathscr{C}^{DI}).
\end{tikzcd}
\end{equation*}
Passing to the canonical mates with respect to the above adjunctions we get an isomorphism
\begin{equation}\label{diag:transfer-iso-large}
\begin{tikzcd}
\h\nerve_h(\mathscr{C}^{DI\times D[0]})\arrow[d, "{(D\pr_1, D\pr_2)^*}"'] \arrow[r, "e"] & \h\nerve_h(\mathscr{C}^{DI})\arrow[d, "\id"]\twocell[dl]\\
\h\nerve_h(\mathscr{C}^{D(I\times[0])}) \arrow[r, "(D\incl)^*"'] &\h\nerve_h(\mathscr{C}^{DI}).
\end{tikzcd}
\end{equation}
\end{constr}

\begin{lemma}\label{lemma:nh-vs-nf-d}
The diagram
\begin{equation*}
\begin{tikzcd}[column sep=0pt]
& \h\nerve_h(\mathscr{C}^{DI\times D[0]})\arrow[rr]\arrow[dd] & & \h\nerve_h(\mathscr{C}^{DI})\arrow[dd]\\
\h\nerve_f(\mathscr{C}_{\textup{R}}^{DI\times D[0]})\arrow[ur]\arrow[rr, crossing over] & & \h\nerve_f(\mathscr{C}_{\textup{R}}^{DI})\arrow[ur]\\
& \h\nerve_h(\mathscr{C}^{D(I\times[0])})\arrow[rr] & & \h\nerve_h(\mathscr{C}^{DI}),\\
\h\nerve_f(\mathscr{C}_{\textup{R}}^{D(I\times[0])})\arrow[from=uu]\arrow[rr]\arrow[ur] & &\h\nerve_f(\mathscr{C}_{\textup{R}}^{DI})\arrow[from=uu, crossing over]\arrow[ur]
\end{tikzcd}
\end{equation*}
where the front face is from (\ref{diag:transfer-iso}) and the back face is (\ref{diag:transfer-iso-large}), is coherent.
\begin{proof}
By the construction of (\ref{diag:transfer-iso}) and since all relevant maps are equivalences it suffices to prove this for 
\begin{equation*}
\begin{tikzcd}[column sep=0pt]
& \Ho(\mathscr{C}^{DI\times D[0]})\arrow[rr]\arrow[dd] & & \Ho(\mathscr{C}^{DI})\arrow[dd]\\
\h\nerve_h(\mathscr{C}^{DI\times D[0]})\arrow[ur]\arrow[rr, crossing over] & & \h\nerve_h(\mathscr{C}^{DI})\arrow[ur]\\
& \Ho(\mathscr{C}^{D(I\times[0])})\arrow[rr] & & \Ho(\mathscr{C}^{DI}).\\
\h\nerve_h(\mathscr{C}^{D(I\times[0])})\arrow[from=uu]\arrow[rr]\arrow[ur] & &\h\nerve_h(\mathscr{C}^{DI})\arrow[from=uu, crossing over]\arrow[ur]
\end{tikzcd}
\end{equation*}
Here the front and back face are filled with the isomorphisms discussed above and all other faces are filled with the respective identities. 
For this it is enough to prove this after passing to canonical mates in $x$-direction (using the adjunctions established before, i.e.~in particular viewing the above maps as \emph{left} adjoints), which is a diagram
\begin{equation*}
\begin{tikzcd}[column sep=0pt]
& \Ho(\mathscr{C}^{DI\times D[0]})\arrow[dd] & &\arrow[ll] \Ho(\mathscr{C}^{DI})\arrow[dd]\\
\h\nerve_h(\mathscr{C}^{DI\times D[0]})\arrow[ur] & &\arrow[ll, crossing over] \h\nerve_h(\mathscr{C}^{DI})\arrow[ur]\\
& \Ho(\mathscr{C}^{D(I\times[0])}) & &\arrow[ll] \Ho(\mathscr{C}^{DI}).\\
\h\nerve_h(\mathscr{C}^{D(I\times[0])})\arrow[from=uu]\arrow[ur] & &\arrow[ll]\h\nerve_h(\mathscr{C}^{DI})\arrow[from=uu, crossing over]\arrow[ur]
\end{tikzcd}
\end{equation*}
with $2$-cells yet to be identified. But by construction the front and back face of the resulting cube are filled with the identity transformation and since all the relevant units and counits are the identities the same is true for the remaining faces and the claim follows.
\end{proof}
\end{lemma}

\begin{lemma}\label{lemma:nh-vs-nh}
The diagram
\begin{equation}\label{diag:the-final-obstacle-easier}
\begin{tikzcd}[column sep=0pt]
& \h\big(\nerve_h(\mathscr{C})^{\nerve I\times \Delta^0}\big)\arrow[rr]\arrow[dd] & & \h\big(\nerve_h(\mathscr{C})^{\nerve I}\big)\arrow[dd]\\
\h\nerve_h(\mathscr{C}^{DI\times D[0]})\arrow[ur]\arrow[rr, crossing over] & & \h\nerve_h(\mathscr{C}^{DI})\arrow[ur]\\
& \h\big(\nerve_h(\mathscr{C})^{\nerve I\times \Delta^0}\big)\arrow[rr] & & \h\big(\nerve_h(\mathscr{C})^{\nerve I}\big),\\
\h\nerve_h(\mathscr{C}^{D(I\times[0])})\arrow[from=uu]\arrow[rr]\arrow[ur] & &\h\nerve_h(\mathscr{C}^{DI})\arrow[from=uu, crossing over]\arrow[ur]
\end{tikzcd}
\end{equation}
where the front face is given by (\ref{diag:transfer-iso-large}), is coherent.
\begin{proof}
Again it suffices to prove this after passing to canonical mates in the $x$-direction (viewing the above as left adjoints), which is strictly commutative by the same argument as above.
\end{proof}
\end{lemma}

Together we get:

\begin{cor}\label{cor:the-final-obstacle}
The diagram
\begin{equation}\label{diag:the-final-obstacle}
\begin{tikzcd}[column sep=0pt]
& \h\big(\nerve_f(\mathscr{C})^{\nerve I\times \Delta^0}\big)\arrow[rr]\arrow[dd] & & \h\big(\nerve_f(\mathscr{C})^{\nerve I}\big)\arrow[dd]\\
\h\nerve_f(\mathscr{C}_{\textup{R}}^{DI\times D[0]})\arrow[ur]\arrow[rr, crossing over] & & \h\nerve_f(\mathscr{C}_{\textup{R}}^{DI})\arrow[ur]\\
& \h\big(\nerve_f(\mathscr{C})^{\nerve I\times \Delta^0}\big)\arrow[rr] & & \h\big(\nerve_f(\mathscr{C})^{\nerve I}\big),\\
\h\nerve_f(\mathscr{C}_{\textup{R}}^{D(I\times[0])})\arrow[from=uu]\arrow[rr]\arrow[ur] & &\h\nerve_f(\mathscr{C}_{\textup{R}}^{DI})\arrow[from=uu, crossing over]\arrow[ur]
\end{tikzcd}
\end{equation}
where the front face is (\ref{diag:transfer-iso}), is coherent.
\begin{proof}
Lemma~\ref{lemma:nf-nh-general} together with Lemma~\ref{lemma:nh-vs-nf-d} reduces this to the coherence of (\ref{diag:the-final-obstacle-easier}), which has been verified in Lemma~\ref{lemma:nh-vs-nh}.
\end{proof}
\end{cor}

\begin{prop}\label{prop:final-equivalence}
The maps $\Phi$ from Construction~\ref{constr:frame-diag} assemble into a natural equivalence of prederivators $\HOND({\mathscr C})\to\HOinf\big(\nerve_f({\mathscr C})\big)$.
\begin{proof}
Obviously $\Phi$ provides a strictly natural transformation of the underlying $1$-functors and it is an equivalence by Proposition~\ref{prop:phi-equiv}. Accordingly it only remains to prove compatibility with $2$-cells.

A trivial calculation shows that the pastings
\begin{equation*}
\begin{tikzcd}
\h\nerve_f(\mathscr{C}_{\textup{R}}^{DI\times D[1]}) \arrow[r, "\Phi"] &[-.5em]
\h\big(\nerve_f(\mathscr{C}_{\textup{R}}^{DI})^{\Delta^1}\big) \arrow[r, "\ev_0"{name=s, above}, bend left=20pt]\arrow[r, "\ev_1"{name=t, below,xshift=2pt}, bend right=20pt]\twocell[from=s, to=t] &
\h\nerve_f(\mathscr{C}_{\textup{R}}^{DI})\arrow[r, "\Phi"] &[-.5em]
\h\big(\nerve_f(\mathscr{C})^{\nerve I})
\end{tikzcd}
\end{equation*}
and
\begin{equation*}
\begin{tikzcd}
\h\nerve_f(\mathscr{C}_{\textup{R}}^{DI\times D[1]}) \arrow[r, "\Phi"] &[-.5em]
\h\big(\nerve_f(\mathscr{C}_{\textup{R}}^{DI})^{\Delta^1}\big) \arrow[r, "\Phi"] &[-.5em]
\h\big(\nerve_f(\mathscr{C})^{\nerve I\times\Delta^1}\big) \arrow[r, "\ev_0"{name=s, above}, bend left=20pt]\arrow[r, "\ev_1"{name=t, below,xshift=2pt}, bend right=20pt]\twocell[from=s, to=t] &
\h\big(\nerve_f(\mathscr{C})^{\nerve I})
\end{tikzcd}
\end{equation*}
agree (even before passing to homotopy categories). To prove compatibility of $\Phi$ with the rest of Construction~\ref{constr:hond-trafo} we apply the same strategy as in the proof of Proposition~\ref{prop:2-cells-natural}. The only non-trivial part to show here is that the diagram (\ref{diag:the-final-obstacle}) is coherent, which is accounted for by Corollary~\ref{cor:the-final-obstacle} above.
\end{proof}
\end{prop}

Again the above obviously constitutes a natural transformation $\HOND\Rightarrow\HOinf\circ\nerve_f$. Hence we get:

\begin{thm}\label{thm:comp-result}
There exists a (up to isomorphism preferred) pseudonatural equivalence
\begin{equation*}
\HOcof\buildrel{\simeq\;}\over\Longrightarrow\HOinf\circ\nerve_f
\end{equation*}
of strict $1$-functors $\cat{COFCAT}\to\cat{PREDER}$.
\begin{proof}
By the above we have a zig-zag
\begin{equation*}
\HOcof\Rightarrow\HOD\Leftarrow\HOND\Rightarrow\HOinf\circ\nerve_f
\end{equation*}
of strictly natural equivalences; here the first equivalence was established in Proposition~\ref{prop:2-cells-natural}, the second one in Proposition~\ref{prop:hond-vs-hod}, and the last one in Proposition~\ref{prop:final-equivalence}.

The claim now follows from Lemma~\ref{lemma:der-quasi-inverse}.
\end{proof}
\end{thm}

\section{Applications}
For the following result a different proof has been previously sketched in \cite{hoder-quasicat}.

\begin{cor}\label{cor:quasi-derivator}
Let $\mathscr{C}$ be a quasi-category. 
\begin{enumerate}
\item If $\mathscr C$ is cocomplete, then $\HOinf(\mathscr C)$ is a right derivator.
\item If $\mathscr C$ is complete, then $\HOinf(\mathscr C)$ is a left derivator.
\item If $\mathscr C$ is complete and cocomplete, then $\HOinf(\mathscr C)$ is a derivator.
\end{enumerate}
\end{cor}
\begin{proof}
For the first statement we use Theorem~\ref{thm:frame-equiv} to get a cofibration category $\mathscr B$ such that $\mathscr C\simeq\nerve_f(\mathscr B)$ and then note $\HOinf(\mathscr C)\simeq\HOinf\big(\nerve_f(\mathscr B)\big)\simeq\HOcof(\mathscr B)$ by Proposition~\ref{prop:hoinf-we} and Theorem~\ref{thm:comp-result}. Hence the claim follows from Theorem~\ref{thm:cof-hoder}.

From this the second statement follows by duality and the third one is obviously a consequence of the other two.
\end{proof}

Next we want to prove that $\HOcof(\mathscr C)$ is strong for any ABC~cofibration category $\mathscr C$. For this we study the case of quasi-categories first:

\begin{lemma}\label{lemma:smothering-1-skeletal}
Let $\mathscr C$ be a quasi-category and let $K$ be a $1$-skeletal simplicial set. Then the forgetful map
\begin{equation}
\h\left({\mathscr C}^K\right)\to(\h\mathscr C)^{\h K}\label{eq:incoherent}
\end{equation}
is full and (essentially) surjective.
\begin{proof}
We first show surjectivity. For this note that we have pushout squares
\begin{equation*}
\begin{tikzcd}
\coprod_{i\in I}\partial\Delta^1 \arrow[r, "\coprod\textup{incl}"] \arrow[dr, phantom, "\ulcorner", very near end] \arrow[d] & \coprod_{i\in I}\Delta^1 \arrow[d]\\
\coprod_{j\in J}\Delta^0 \arrow[r] & K
\end{tikzcd}
\quad\text{and}\quad
\begin{tikzcd}
\coprod_{i\in I}\partial[1] \arrow[r, "\coprod\textup{incl}"] \arrow[dr, phantom, "\ulcorner", very near end] \arrow[d] & \coprod_{i\in I}[1] \arrow[d]\\
\coprod_{j\in J}[0] \arrow[r] & \h K
\end{tikzcd}
\end{equation*}
for some sets $I,J$: namely, the first pushout comes from the assumption on $K$ and the second one comes from the fact that $\h$ is a left adjoint. Now the objects of the left hand side of (\ref{eq:incoherent}) are precisely simplicial maps $K\to\mathscr C$ and the objects of the right hand side are functors $\h K\to\h\mathscr C$. Hence by the above pushouts it suffices to consider the special case $K=\Delta^1$ (the preimages will automatically fit together since the vertices of $\mathscr C$ are precisely the objects of $\h\mathscr C$). However, an object of $(\h\mathscr C)^{[1]}$ is a morphism in $\h\mathscr C$ which is represented by an edge of $\mathscr C$, cf.~e.g.~\cite[Proposition 1.2.3.9]{htt}.

For fullness we note that both $\blank\times\Delta^1$ and $\blank\times[1]$ are left-adjoints, hence we have pushouts
\begin{equation}\label{diag:prod}
\!\!
\begin{tikzcd}[column sep=scriptsize]
\coprod_{i\in I}(\partial\Delta^1)\times\Delta^1 \arrow[r] \arrow[dr, phantom, "\ulcorner", very near end] \arrow[d] & \coprod_{i\in I}\Delta^1\times\Delta^1 \arrow[d]\\
\coprod_{j\in J}\Delta^0\times\Delta^1 \arrow[r] & K\times\Delta^1
\end{tikzcd}
\;
\begin{tikzcd}[column sep=scriptsize]
\coprod_{i\in I}(\partial[1])\times[1] \arrow[r] \arrow[dr, phantom, "\ulcorner", very near end] \arrow[d] & \coprod_{i\in I}[1]\times[1] \arrow[d]\\
\coprod_{j\in J}[0]\times[1] \arrow[r] & (\h K)\times[1]
\end{tikzcd}
\!\!
\end{equation}
Now simplicial maps $K\times\Delta^1\to\mathscr C$ describe precisely the edges of ${\mathscr C}^K$ and functors $(\h K)\times[1]\to\h\mathscr C$ form precisely the morphisms in the right hand side. Now let $f,g\colon K\to\mathscr C$ be simplicial maps and let $\tau\colon \h f\Rightarrow\h g$ be a natural transformation. We choose for each object $x$ of $\h K$ an edge $e_x$ in $\mathscr C$ representing $\tau_x$. Since any $1$-simplex of ${\mathscr C}^K$ defines a morphism in $\h(\mathscr C^K)$ it suffices to extend these together with $f$ and $g$ to a simplicial map $K\times\Delta^1\to\mathscr C$. By the pushouts (\ref{diag:prod}) it suffices to consider the case $K=\Delta^1$ again (compatibility of the preimages is guaranteed because we lifted on the boundary \emph{before}). But this precisely means, that we have to show that a ($1$-dimensional) diagram
\begin{equation*}
\begin{tikzcd}
a \arrow[r, "f"]\arrow[d, "g"'] & b \arrow[d, "h"]\\
c \arrow[r, "i"'] & d
\end{tikzcd}
\end{equation*}
in $\mathscr{C}$ that commutes in $\h\mathscr{C}$ extends to a square in $\mathscr{C}$. Indeed, choose an edge $e$ in $\mathscr{C}$ representing $[h][f]=[i][g]$. Then we have by construction $2$-simplices $(f,h,e)$ and $(g,i,e)$ which precisely provide the desired extension.
\end{proof}
\end{lemma}

In order to apply this to our case we will use:

\begin{lemma}\label{lemma:nh-unit}
Let $K$ be a $1$-skeletal simplicial set. Then the unit $K\to \nerve(\h K)$ is a categorical equivalence.
\begin{proof}
While the above statement is a rather direct consequence of the Quillen equivalence \cite[Theorem 2.2.5.1]{htt}, one can prove by more elementary means that the map in question is even inner anodyne, which we shall do now.

We first remark that edges $X\to Y$ in $\nerve(\h K)$, i.e.~morphisms $X\to Y$ in $\h K$, correspond bijectively to sequences
\begin{equation*}
\begin{tikzcd}
X = X_0 \arrow[r, "e_1"] & X_1 \arrow[r, "e_2"] & \cdots \arrow[r, "e_r"]& X_r = Y
\end{tikzcd}
\end{equation*}
of adjacent non-degenerate edges in $K$; we call $r$ the \emph{rank} of the edge. More generally let us define the rank of any $n$-simplex $\sigma$ to be the rank of its \emph{long edge} $\sigma|_{\Delta^{\{0,n\}}}$ (with the convention that $\sigma|_{\Delta^{\{0,n\}}}$ is the degenerate edge at $\sigma$ if $n=0$).
The unit $K\to\nerve(\h K)$ can then be identified with the inclusion of the simplicial subset of simplices of rank at most $1$.

Let us call an $n$-simplex $\sigma\in\nerve(\h K)_n$ \emph{primitive} if all the edges $\sigma|_{\Delta^{\{i,i+1\}}}$ are given by non-degenerate edges of $K$, i.e.~have rank $1$.

\begin{claim}
Let $\sigma\in\nerve(\h K)_r$ be a simplex of rank $n$. Then there exists a unique pair of a primitive $n$-simplex $\tau$ and a monotone map $f\colon [r]\to[n]$ such that $\sigma=f^*\tau$.
\begin{proof}
We begin by recalling that a simplex of positive dimension in the nerve of any category is uniquely characterized by its restrictions to the edges $\Delta^{\{i,i+1\}}$. It follows in our special case that any simplex of $\nerve(\h K)$ is uniquely characterized by its long edge and the ranks of each of the edges $\Delta^{\{i,i+1\}}$. In particular, a primitive simplex is uniquely characterized by its long edge; conversely, it is obvious that any morphism of $\h K$ appears as long edge of some primitive simplex.

We can now prove uniqueness. Let $\tau$ be any primitive $n$-simplex through which $\sigma$ factors. Since the long edge of $\sigma$ has rank $n$ it has to coincide with the long edge of $\tau$. Hence the above implies the uniqueness of $\tau$. Now let $g\colon[s]\to[n]$ be any monotone map and $0\le i\le j\le n$. Then $(g^*\tau)|_{\Delta^{\{i,j\}}}$ is easily seen to have rank $g(j)-g(i)$. Applying this to $f\colon[r]\to[n]$ with $\sigma=f^*\tau$ implies that $f(r)-f(0)=n$ which is only possible if $f(0)=0$. But then the same argument shows that $f(i)=f(i)-f(0)$ equals the rank of $\sigma|_{\Delta^{\{0,i\}}}$ also~for $i>0$, proving uniqueness of $f$.

For the existence proof the above dictates what to do: we take $\tau$ to be the unique primitive $n$-simplex whose long edge is the long edge of $\sigma$. We moreover define $f\colon[r]\to[n]$ via $f(i)=\text{rank of $\sigma|_{\Delta^{\{0,i\}}}$}$. Then one immediately sees that $f^*\tau$ and $\sigma$ have the same long edge and that for each $0\le i<n$ the restrictions $(f^*\tau)|_{\Delta^{\{i,i+1\}}}$ and $\sigma|_{\Delta^{\{i,i+1\}}}$ have the same rank. Accordingly the above implies $f^*\tau=\sigma$, finishing the proof of the claim.
\end{proof}
\end{claim}

Let us now define $K^{(n)}\subset\nerve(\h K)$ to be the simplicial subset of those simplices of rank at most $n$. Then we have
\begin{equation*}
K\cong K^{(1)}\subset K^{(2)}\subset\cdots\subset K^{(n)}\subset\cdots\subset\nerve(\h K)
\end{equation*}
and moreover obviously $\nerve(\h K)=\bigcup_{n\ge 1} K^{(n)}$. Accordingly it suffices to prove that each of the inclusions $K^{(n-1)}\to K^{(n)}$ is inner anodyne.

Denote by $X_n$ the set of primitive $n$-simplices of $K^{(n)}$. We now claim that the square
\begin{equation}\label{diag:pushout-primitive}
\begin{tikzcd}
X_n\times\Lambda^{\{1,\dots,n-1\}}[n]\arrow[d]\arrow[r] & X_n\times\Delta^n\arrow[d]\\
K^{(n-1)} \arrow[r] & K^{(n)},
\end{tikzcd}
\end{equation}
where the horizontal maps are the inclusions and the vertical maps are the tautological ones, is a pushout. Here $\Lambda^{\{1,\dots,n-1\}}[n]$ is one of the generalized inner horns in the sense of \cite[Section 2.2.1]{joyal}, namely the simplicial subset of $\Delta^n$ given by all simplices not containing the long edge.

Indeed, denote the pushout of the above span by $P$. Then the induced map $\alpha\colon P\to K^{(n)}$ is surjective: let $\sigma$ be any simplex of $K^{(n)}$. If it has rank strictly less than $n$, then $\sigma\in K^{(n-1)}$ by definition. Otherwise the claim provides us with a primitive $n$-simplex $\tau$ such that $\sigma$ is in the image of $\Delta^n_\tau$, the copy of $\Delta^n$ corresponding to $\tau\in X_n$.

But $\alpha$ is also injective: indeed, let $\sigma_1,\sigma_2\in P_m$ be such that $\alpha(\sigma_1)=\alpha(\sigma_2)$. If both of them lie in $K^{(n-1)}\subset P$, we are done because the restriction of $\alpha$ to $K^{(n-1)}$ is injective by construction. Accordingly we may assume without loss of generality that $\sigma_1\notin K^{(n-1)}$. 
 By construction of the pushout this means that there exists a $\tau_1\in X_n$ and a monotone map $f_1\colon[m]\to[n]$ such that $\sigma_1$ is given by $f_1^*\Delta^n_{\tau_1}$, and moreover $\sigma_1$ has to contain the long edge of $\Delta^n_{\tau_1}$. This implies that $\alpha(\sigma_1)$ contains the long edge of the primitive $n$-simplex $\tau_1$ of $K^{(n)}$. But $\alpha(\sigma_2)=\alpha(\sigma_1)$ and hence also $\alpha(\sigma_2)$ has to contain this edge. Since any edge of $K^{(n-1)}$ has rank strictly less than $n$, the same argument as above yields $\tau_2\in X_n$ and a monotone map $f_2\colon[m]\to[n]$ such that $\sigma_2$ is given by $f_2^*\Delta^n_{\tau_2}$. But then $f_1^*\tau_1=\alpha(\sigma_1)=\alpha(\sigma_2)=f_2^*\tau_2$ in $K^{(n)}$. Since $\alpha(\sigma_1)=\alpha(\sigma_2)$ is a simplex of rank $n$ and $\tau_1,\tau_2$ are primitive $n$-simplices, the claim allows us to conclude $\tau_1=\tau_2$ and $f_1=f_2$. It follows $\sigma_1=\sigma_2$ as desired.

Hence (\ref{diag:pushout-primitive}) is a pushout square. But its top map is inner anodyne by \cite[Proposition~2.12-(iv) and Theorem~2.17]{joyal} and hence so is $K^{(n-1)}\to K^{(n)}$ as a pushout of an inner anodyne map, finishing the proof.
\end{proof}
\end{lemma}

\begin{prop}\label{prop:hoinf-strong}
Let $\mathscr{C}$ be a quasi-category. Then the prederivator $\HOinf(\mathscr{C})$ is strong.
\begin{proof}
Since $\HOinf(\mathscr{C})^A\cong\HOinf(\mathscr{C}^{\nerve A})$ for any small category $A$ it suffices to show that for each free category $F$ the underlying diagram functor $\diag\colon\HOinf(\mathscr{C})(F)\to\HOinf(\mathscr{C})(*)^F$ is full and essentially surjective. For this we observe that it factors as the composition
\begin{equation*}
\h\big(\mathscr{C}^{\nerve F})\to(\h\mathscr{C})^{\h\nerve F}\to (\h\mathscr{C})^F\cong\HOinf(\mathscr{C})(*)^F
\end{equation*}
where the left hand functor is the forgetful map and the middle map is given by restriction along the inverse of the counit of $\h\dashv\nerve$.

Now picking any isomorphism $\varphi\colon \h K\to F$ we get a commutative diagram
\begin{equation*}
\begin{tikzcd}
\h\big(\mathscr{C}^{\nerve F}\big) \arrow[d, "\cong", "(\nerve\varphi)^*"'] \arrow[r] \arrow[rr, "\diag", bend left=20] & (\h\mathscr{C})^{\h\nerve F} \arrow[r, "\cong"'] \arrow[d, "(\h\nerve\varphi)^*", "\cong"'] & \HOinf(\mathscr{C})(*)^F\\
\h\big(\mathscr{C}^{\nerve\h K}\big) \arrow[d,"\simeq", "\eta^*"'] & (\h\mathscr{C})^{\h\nerve\h K}\arrow[d, "(\h\eta)^*", "\simeq"']\\
\h\big(\mathscr{C}^K\big) \arrow[r] & (\h\mathscr{C})^{\h K}
\end{tikzcd}
\end{equation*}
where the lower vertical maps are equivalences by Lemma~\ref{lemma:nh-unit}. Since the lower horizontal map is full and essentially surjective by Lemma~\ref{lemma:smothering-1-skeletal}, the claim follows.
\end{proof}
\end{prop}

With this we can prove:

\begin{cor}\label{cor:hocof-strong}
Let $\mathscr{C}$ be an ABC cofibration category (for example a model category). Then the prederivator $\HOcof(\mathscr{C})$ is strong.
\begin{proof}
Corollary~\ref{cor:cof-approx} and Theorem~\ref{thm:comp-result} provide equivalences 
\begin{equation*}
\HOcof(\mathscr{C})\simeq\HOcof(\mathscr{C}_c)\simeq\HOinf\big(\nerve_f(\mathscr{C}_c)\big)
\end{equation*}
and hence the claim follows from Proposition~\ref{prop:hoinf-strong}.
\end{proof}
\end{cor}

Finally we note:

\begin{cor}\label{cor:we-diag}
The functor $\HOcof$ preserves and reflects weak equivalences of cofibration categories.
\begin{proof}
Weak equivalences are reflected by definition. For the second statement it suffices to note that both $\HOinf$ (by Proposition~\ref{prop:hoinf-we}) and $\nerve_f$ (as an exact functor) preserve weak equivalences, and then apply Theorem~\ref{thm:comp-result} again.
\end{proof}
\end{cor}

We think that the above should have already been known before, but we do not know of an explicit reference. Here is a sketch of an alternative proof of the non-trivial part: by \cite[Corollaire 3.20${}^\op$]{approximation} the above is true when we restrict ourselves to finite direct categories as index categories and the proof actually works for finite direct categories with weak equivalences. In the presence of the additional two axioms it generalizes further to all small direct categories with weak equivalences. Now one can use Proposition~\ref{prop:cof-hoder-models}-(\ref{item:p-star}) to conclude as above.

\begin{rk}
Also the functor $\HOinf$ reflects equivalences between (arbitrary) quasi-categories, but the proof of this requires different means.
\end{rk}

\frenchspacing
\bibliographystyle{alpha}
\bibliography{literature.bib}
\end{document}